\documentclass[12pt]{article}
\usepackage{amsmath,amssymb,amsbsy,amsfonts,amsthm,latexsym,
            amsopn,amstext,amsxtra,euscript,amscd,amsthm,url}
\usepackage[margin=3cm]{geometry}
\usepackage[usenames, dvipsnames]{color}
\usepackage[utf8]{inputenc}
\usepackage[english]{babel}
\usepackage{mathrsfs}
\usepackage{caption}
\usepackage{xcolor}

\newtheorem{thm}{Theorem}
\newtheorem{rem}{Remark}

\newtheorem{lem}{Lemma}
\newtheorem{cor}{Corollary}

\DeclareMathOperator{\1}{\textbf{1}}
\newcommand{\Ocal}{O}

\DeclareMathOperator{\id}{id}
\global\long\def\epsilon{\varepsilon} 
\usepackage{amsmath}
\usepackage[most]{tcolorbox}
\DeclareMathAlphabet{\mathpzc}{OT1}{pzc}{m}{it}

\tcbset{colback=black!2!white, colframe=black!50!black, 
        highlight math style= {enhanced, 
            colframe=black,colback=black!2!white,boxsep=0pt}
}
\begin{document}
\date{}

\title{On sums of weighted averages of $\gcd$-sum functions}
\author{Isao Kiuchi and Sumaia Saad Eddin}

\maketitle
{\def\thefootnote{}
\footnote{{\it Mathematics Subject Classification 2010: 11A25, 11N37, 11Y60.}}

\begin{abstract}
Let $\gcd(j,k)$ be the greatest common  divisor of the integers $j$ and $k$.
In this paper, we give  several interesting  asymptotic formulas  for  weighted averages  of the $\gcd$-sum function 
$
f(\gcd(j,k)) 
$
and the function 
$
\sum_{d|k, d^{s}|j}(f*\mu)(d)  
$
for any positive  integers $j$ and $k$,   namely  
$$
\sum_{k\leq x}\frac{1}{k^{r+1}}\sum_{j=1}^{k}j^{r}f(\gcd(j,k))
\quad
\text{and}
\quad  
\sum_{k\leq x}\frac{1}{k^{s(r+1)}}\sum_{j=1}^{k^s}j^{r} \sum_{\substack{d|k \\ d^{s}|j}}(f*\mu)(d) 
$$
with any fixed  integer $s> 1$ and  any arithmetical function $f$.
We also  establish  mean value formulas for the error terms  
of asymptotic formulas  for partial sums of  $\gcd$-sum functions 
$
f(\gcd(j,k)). 
$
\end{abstract}

\section{Introduction and statements of  the  results }                                

\subsection{Weighted averages of $\gcd$-sum functions }
Let $\gcd(j,k)$ be the greatest common  divisor of the integers $j$ and $k$. The $\gcd$-sum function,  which is also known as Pillai's arithmetical function,  
is defined by 
$$
P(n) = \sum_{k=1}^{n}\gcd(k,n).
$$
Properties and various generalizations of $P(n)$ have been widely studied by many authors.
For a nice survey on it see \cite{To1}.  Let 
$$
P_{f}(n)=\sum_{k=1}^{n}f(\gcd(k,n))
$$
be one of many generalizations of $P(n)$, for an arbitrary arithmetical function $f$.
In 2010,  Bordell\'{e}s \cite{Bo1} provided  some general asymptotic formulas for the 
partial sum of $P_{f}(n)$ with $f$ belonging to certain classes of arithmetic functions.
More recently, the sum of  the weighted average of the $\gcd$-sum functions with various 
completely multiplicative functions  as weights was   first  considered  by the first author  \cite{K1}. 
For any fixed positive integer $r$ and any arithmetical function $f$,  he  proved   that   
\begin{align}                                     \label{ikiuchi}
M_{r}(x;f)& := \sum_{k\leq x} \frac{1}{k^{r+1}}\sum_{j=1}^{k}j^{r}f(\gcd(j,k)) \\ 
& = \frac12  \sum_{n\leq x}\frac{f(n)}{n}   + \frac{1}{r+1}\sum_{dl\leq x}\frac{(f*\mu)(d)}{d}   \nonumber \\
& \qquad  \qquad  + \frac{1}{r+1}\sum_{m=1}^{[r/2]}\binom{r+1}{2m} B_{2m} \sum_{dl\leq x}\frac{(f*\mu)(d)}{d} \frac{1}{l^{2m}}.    \nonumber 
\end{align} 
Here  $\mu$, as usual, denotes the M\"{o}bius function and   $B_{m}(x)$  is   the  Bernoulli  polynomials 
defined by the generating  function 
$$
\frac{ze^{xz}}{e^{z}-1}=\sum_{m=0}^{\infty}B_{m}(x)\frac{z^m}{m!},
$$
with $|z|<  2\pi$, 
where $B_m$ is the Bernoulli number given by $B_{m}(0)$ (see \cite{Ap}, \cite{Coh}). \\
Let $\tau$ be the divisor function defined by $\1*\1$, and  let $\theta$ be the number appearing in the Dirichlet divisor problem, namely 
\begin{equation} 
\label{Dirichlet}
\sum_{n\leq x}\tau(n) = x\log x + (2\gamma -1)x + \Delta(x),                                  
\end{equation}
with $\Delta(x)=\Ocal \left(x^{\theta+\varepsilon}\right)$ for any positive number $\varepsilon$.
Many applications of (\ref{ikiuchi}) have been given in \cite{K1}.
For example, with $f=\id$, the first author  proved that
\begin{multline}                                                                                                 \label{K-id}
 M_{r}(x;\id)   
  = \frac{1}{(r+1)\zeta(2)} x\log x + \frac{x}{2}   \\
 + \frac{1}{(r+1)\zeta(2)}\left(2\gamma -1 - \frac{\zeta'(2)}{\zeta(2)} + \sum_{m=1}^{[r/2]}\binom{r+1}{2m} B_{2m} \zeta(2m+1) \right)x  
 + K_{r}(x), 
\end{multline}
where $\zeta$ denotes the Riemann zeta-function and   
\begin{equation}                                                                                                 \label{K-K}
  K_{r}(x) 
= \frac{1}{r+1} \sum_{n\leq x}\frac{\mu(n)}{n}\Delta\left(\frac{x}{n}\right) + \Ocal_{r}\left(\log x\right).
\end{equation}
For any real or complex number $a$, 
$\sigma_{a}$  denotes  a generalized  divisor function defined by $\id_{a}*{\bf 1}$,
 where $ \id_{a}(n)=n^a$ for any positive integer $n$. 
We recall that   
\begin{equation}                                                                                               \label{Delta-a}
\sum_{n\leq x}\sigma_{a}(n) =  \zeta(1-a)x +  \frac{\zeta(1+a)}{1+a} x^{1+a} - \frac{\zeta(-a)}{2} + \Delta_{a}(x),   
\end{equation}
for  $-1 < a < 0$. The problem of improving  $\Delta_{a}(x)$ is known as the generalized Dirichlet divisor problem.  
In 1988,   P\'{e}termann~\cite{Pe} showed that 
\begin{equation}                                                                                                \label{Peter}
\Delta_{a}(x) = 
\Ocal \left(x^{\frac{1+a}{3}+\varepsilon}\right),
\end{equation}
for  any  small  number $\varepsilon>0$.  
In this direction of research, the first author~\cite{k} derived the following mean value formula, for $-\frac12 <a<0$,  
\begin{equation}                                                                                               \label{K-mean}
 \int_{0}^{T}\Delta_{a}(x)^{2}dx  
= \frac{\zeta(\frac{3}{2} -a)\zeta(\frac{3}{2} +a)\zeta(\frac{3}{2})^{2}}{\zeta(3)(6+4a)\pi^2} T^{\frac{3}{2}+a} 
+ \Ocal\left(T^{\frac{5}{4}+\frac{a}{2}+\varepsilon}\right) + \Ocal\left(T^{1+\varepsilon}\right).
\end{equation}
The first purpose of this paper is to give   an  asymptotic formula of $M_{r}(x;\id_{1+a})$ for $-1<a<0$.  
We prove that 
\begin{thm} 
 \label{th1}
 For  any  real number $x >1$,  
fixed  positive  integer $r$  and  
fixed  number $a$ such that $-1<a<0$,  we have
\begin{multline}                                   \label{sum_th1}
 M_{r}(x;\id_{1+a})    
 =\frac{\zeta(1-a)}{(r+1)\zeta(2)} x  + \frac{1}{2(1+a)}x^{1+a} \\
 + \frac{x^{1+a}}{(a+1)(r+1)\zeta(2+a)} 
\left(\zeta(1+a) + \sum_{m=1}^{[r/2]}\binom{r+1}{2m} B_{2m} \zeta(2m+a+1)\right)
  \\+ K_{r}(x;\id_{1+a}),                       
\end{multline}
where  
\begin{equation}                                  \label{K-S-th1}
  K_{r}(x;\id_{1+a}) 
= \frac{1}{r+1} \sum_{n\leq x}\frac{\mu(n)}{n}\Delta_{a}\left(\frac{x}{n}\right) + \Ocal_{r,a}\left(1\right).
\end{equation}
\end{thm}
Note that \eqref{K-S-th1} is an analogue of \eqref{K-K}.  
Using the elementary results  of $\Delta(x)$ and $\Delta_{a}(x)$,
the functions $K_{r}(x)$ and $K_{r}(x;\id_{1+a})$ are  estimated  by 
$\Ocal_{r}\left(x^{1/2 +\varepsilon}\right)$  and $\Ocal_{r}\left(x^{\frac{1+a}{3}+\varepsilon}\right)$, respectively.   
Hence,   we  get 
$$
\lim_{x\to\infty}\frac{M_{r}(x;\id)}{x\log x} = \frac{1}{(r+1)\zeta(2)}
\quad \text{and} \quad 
\lim_{x\to\infty}\frac{M_{r}(x;\id_{1+a})}{x} = \frac{\zeta(1-a)}{(r+1)\zeta(2)}.
$$\\

Let $\phi$ be the Euler totient function defined by $\phi=\id*\mu$.    
As a second  application of~\eqref{ikiuchi} 
with  $f=\phi$, the first author showed that   
\begin{multline}                                                                                                   \label{K-phi}
M_{r}(x;\phi)    
 =\frac{1}{(r+1)\zeta^{2}(2)}x \log x   \\
 + \frac{1}{(r+1)\zeta^{2}(2)}\left(2\gamma -1 -2\frac{\zeta'(2)}{\zeta(2)} + \sum_{m=1}^{[r/2]}\binom{r+1}{2m} B_{2m} \zeta(2m+1)\right) x \\
  + \frac{x}{2\zeta(2)} + L_{r}(x),  
\end{multline}
 where 
\begin{equation}                                                                                                        \label{K-phi1}  
  L_{r}(x) 
:= \frac{1}{r+1}\sum_{n\leq x}\frac{(\mu*\mu)(n)}{n}\Delta\left(\frac{x}{n}\right) + \Ocal_{r}\left((\log x)^2\right). 
\end{equation}
On taking the Dedekind function $\psi\ (=\id*|\mu|)$ in place of the Euler totient function $\phi$ into  
(\ref{ikiuchi}),  he also deduced 
\begin{multline}                                                                                                         \label{gcd331}
M_{r}(x;\psi)  =  \frac{1}{(r+1)\zeta^{}(4)} x\log x      \\
+ 
 \frac{1}{(r+1)\zeta(4)} 
\left(2\gamma - 1 - 2 \frac{\zeta'(4)}{\zeta(4)}+ \sum_{m=1}^{[r/2]}\binom{r+1}{2m}B_{2m} \zeta(2m+1) \right)x  \\
+ \frac{\zeta(2)}{2\zeta(4)} x  + U_{r}(x),  
\end{multline}
 where  
\begin{equation}                                                                                                             \label{Ur}
U_{r}(x) :=  \frac{1}{r+1} \sum_{n\leq x} \frac{(\mu*|\mu|)(n)}{n}\Delta\left(\frac{x}{n}\right)+  \Ocal_{r}\left((\log x)^{2}\right).
\end{equation}

Let $\phi_{\alpha}$ and $\psi_{\beta}$ denote the Jordan totient function and the generalized Dedekind function
defined by   $\id_{\alpha}*\mu$ and $\id_{\beta}*|\mu|$, respectively. 
Our  second purpose  is to  provide  two  asymptotic formulas  of the partial sums of  weighted averages of 
$\phi_{1+a}(\gcd(k,j))$ and $\psi_{1+a}(\gcd(k,j))$  for any fixed negative number $a\ (-1<a<0)$.    
We prove  that 
\begin{thm}  \label{th2}
 For any  real number $x >1$,   
          fixed positive  integer $r$ and  
          fixed number $a$ such that  $-1<a<0$,     we have
\begin{multline}                                  \label{sum_th2}
 M_{r}(x;\phi_{1+a})    
 = \frac{\zeta(1-a)}{(r+1)\zeta(2)^2} x  + \frac{1}{2(1+a)\zeta(2+a)}x^{1+a} \\
+ \frac{x^{1+a}}{(a+1)(r+1)\zeta(2+a)^2} 
\left(\zeta(1+a) + \sum_{m=1}^{[r/2]}\binom{r+1}{2m} B_{2m} \zeta(2m+a+1)\right) \\
 + K_{r}(x;\phi_{1+a})                     
\end{multline}
and 
\begin{multline}                                                                                                     \label{sum_th22}
M_{r}(x;\psi_{1+a})    
 = \frac{\zeta(1-a)}{(r+1)\zeta(4)} x  + \frac{\zeta(2+a)}{2(1+a)\zeta(4+2a)}x^{1+a}  \\
+ \frac{x^{1+a}}{(a+1)(r+1)\zeta(4+2a)} 
\left(\zeta(1+a) + \sum_{m=1}^{[r/2]}\binom{r+1}{2m} B_{2m} \zeta(2m+a+1)\right)  \\
 + K_{r}(x;\psi_{1+a}),                                        
\end{multline}
where   
\begin{equation}                                                                                                    \label{K-S-th2}
  K_{r}(x;\phi_{1+a}) 
= \frac{1}{r+1} \sum_{n\leq x}\frac{(\mu*\mu)(n)}{n}\Delta_{a}\left(\frac{x}{n}\right) + \Ocal_{r,a}\left((\log x)^2\right) 
\end{equation}
and   
\begin{equation}                                                                                                 \label{K-S-th22}
  K_{r}(x;\psi_{1+a}) 
= \frac{1}{r+1} \sum_{n\leq x}\frac{(|\mu|*\mu)(n)}{n}\Delta_{a}\left(\frac{x}{n}\right) + \Ocal_{r,a}\left((\log x)^2\right). 
\end{equation}
\end{thm}
Using \eqref{Peter},  we see that \eqref{K-S-th2} and \eqref{K-S-th22} are estimated by    
$
\Ocal_{r,a}\left(x^{\frac{1+a}{3}+\varepsilon}\right).
$
This leads to the following result.   
\begin{cor} 
Under the hypothesis of Theorem~\ref{th2},  we have 
$$
\lim_{x\to \infty}\frac{ M_{r}(x;\phi_{1+a})}{x} = \frac{\zeta(1-a)}{(r+1)\zeta^{2}(2)}, 
\quad 
\quad  
\lim_{x\to \infty}\frac{M_{r}(x;\psi_{1+a})}{x}  = \frac{\zeta(1-a)}{(r+1)\zeta(4)},    
$$
and 
\begin{multline}
\label{Remark_th2}
\zeta(2+a)^{2} M_{r}(x;\phi_{1+a}) - \zeta(4+2a)M_{r}(x;\psi_{1+a})\\  =  \frac{\zeta(1-a)}{r+1}\left(\frac{\zeta(2+a)^2}{\zeta(2)^2}-\frac{\zeta(4+a)}{\zeta(4)}\right) x    \\
 +  \zeta(2+a)^{2}K_{r}(x;\phi_{1+a}) - \zeta(4+2a) K_{r}(x;\psi_{1+a}). 
\end{multline}
\end{cor}
This latter formula follows at once from \eqref{sum_th2} and \eqref{sum_th22}. 
\subsection{Mean value theorems} 
The mean value formulas of $L_{r}(x)$ and $U_{r}(x)$, defined by \eqref{K-phi1} and \eqref{Ur}, were first considered by the first author~\cite{K1}. For any fixed positive  integer $r$, he showed that 
\begin{align}                                                                                                    \label{Lrx}
\int_{1}^{T}L_{r}(x)^{2}dx = \frac{D_{1}}{(r+1)^2} T^{3/2} + \Ocal_{r}\left(T^{5/4+\varepsilon}\right),
\end{align}
and 
\begin{align}                                                                                                    \label{Urx}
\int_{1}^{T}U_{r}(x)^{2}dx = \frac{D_{2}}{(r+1)^2} T^{3/2} + \Ocal_{r}\left(T^{5/4+\varepsilon}\right),
\end{align}
 where 
$$
D_{1}=\frac{1}{6\pi^2}\sum_{n=1}^{\infty}\left(\left(\frac{\mu*\mu}{\sqrt{\id}}*\tau\right)(n)\right)^{2}\frac{1}{n^{3/2}}
$$
and 
$$
D_{2}=\frac{1}{6\pi^2}\sum_{n=1}^{\infty}\left(\left(\frac{\mu*|\mu|}{\sqrt{\id}}*\tau\right)(n)\right)^{2}\frac{1}{n^{3/2}}. 
$$
Our aim in this paper is to establish a mean value formula of $K_{r}(x;{\id}_{1+a})$,  $K_{r}(x;\phi_{1+a})$ and $K_{r}(x;\psi_{1+a}),$  
defined by \eqref{K-S-th1}, \eqref{K-S-th2} and \eqref{K-S-th22},  respectively. Our results are precisely  the following:     

\begin{thm}  \label{th3}
 For any real number $T >1$, fixed positive integer  $r$, and fixed number $a$ such that $-\frac{1}{4} < a < 0$,  we have
\begin{equation}                                                                                             \label{sum_th3}     
\int_{1}^{T}K_{r}(x;\id_{1+a})^{2}dx 
= \frac{C_{2,a}}{(r+1)^{2}} T^{\frac{3}{2}+a}  
 + \Ocal_{r}\left(T^{\frac{5}{4}+\varepsilon}\right) + \Ocal_{r,a}\left(T^{1-a+\varepsilon}\right),    
\end{equation}
where 
$$
C_{2,a} = \frac{1}{2(3+2a)\pi^2} \sum_{n=1}^{\infty}\left(\left(\frac{\mu}{\id_{\frac{1+a}{2}}}*\sigma_{a}\right)(n)\right)^{2}\frac{1}{n^{3/2}}.  
$$
\end{thm}
\begin{thm}  \label{th4}
 For any  real number $T >1$, fixed positive integer  $r$, and fixed number $a$ such that  $-\frac14 < a < 0$, we have
\begin{equation}                                                                                                    
 \label{sum_th4}
\int_{1}^{T}K_{r}(x;\phi_{1+a})^{2}dx = 
  \frac{C_{3,a}}{(r+1)^{2}}  T^{\frac{3}{2}+a}  
 + \Ocal_{r}\left(T^{\frac{5}{4}+\varepsilon}\right) + \Ocal_{r,a}\left(T^{1-a+\varepsilon}\right),                                     
\end{equation}
and 
\begin{equation}                                                                                                     \label{sum_th44}
 \int_{1}^{T}K_{r}(x;\psi_{1+a})^{2}dx = 
  \frac{C_{4,a}}{(r+1)^{2}}  T^{\frac{3}{2}+a}  
 + \Ocal_{r}\left(T^{\frac{5}{4}+\varepsilon}\right)+ \Ocal_{r,a}\left(T^{1-a+\varepsilon}\right),      
\end{equation}
where 
$$
C_{3,a} = \frac{1}{2(3+2a)\pi^2} \sum_{n=1}^{\infty}
\left(\left(\frac{\mu*\mu}{\id_{\frac{1+a}{2}}}*\sigma_{a}\right)(n)\right)^{2}\frac{1}{n^{3/2}},   
$$
and 
$$
C_{4,a} = \frac{1}{2(3+2a)\pi^2} \sum_{n=1}^{\infty}
\left(\left(\frac{\mu*|\mu|}{\id_{\frac{1+a}{2}}}*\sigma_{a}\right)(n)\right)^{2}\frac{1}{n^{3/2}}.    
$$
\end{thm}
From the above, it seems  difficult to improve the $\Ocal$-terms  in our theorem,  since  the error terms of \eqref{K-S-th2} and \eqref{K-S-th22} are in a  weak form. In view of this, it is useful to try to determine a sharper form of \eqref{K-S-th2} and \eqref{K-S-th22}.
\subsection{Weighted averages of Cohen-Ramanujan  sums} 

For a positive integer $s$, the generalized $\gcd$ function  $(j,k^{s})_{s}$ is defined to give the largest $d\in\mathbb{N}$ such that $d|k$ and $d^{s}|j$.
Therefore, $(k,j^{1})_{1}=(k,j)$ is the usual $\gcd$ of two integers. 
Let $c_{k}(j)$ denote the  Ramanujan sum defined by 
$$ 
c_{k}(j)=\sum_{d|\gcd(j,k)}d\mu\left(\frac{k}{d}\right).  
$$  
This function was first introduced by Ramanujan in 1918. In more recent years, various generalizations of the Ramanujan sum have been constructed. 
One of the most known generalizations of $c_{k}(j)$ was given by Cohen \cite{C1},\cite{C2},\cite{C3} defined as follows
\begin{align*}
c_{k}^{(s)}(j)& := \sum_{d^s |(j,k^{s})_{s}}d^{s}\mu\left(\frac{k}{d}\right).  
\end{align*} 
In 2017,  Namboothiri \cite{N} derived certain  identities for weighted averages of  Cohen sums 
with weights concerning logarithms, the Gamma function and the Bernoulli polynomials, and others. 
As another generalization of the Ramanujan sums, Anderson and  Apostol \cite{AA}  introduced the function 
 $s_{k}(j)$ defined by the identity   
$$
s_{k}^{}(j)=\sum_{d|\gcd(k,j)}f(d)g\left(\frac{k}{d}\right) 
$$ 
with  any arithmetical functions $f$ and $g$. This latter is generalized to  
\begin{equation}                                  \label{ASDFG}
s_{k}^{(s)}(j) := \sum_{d^s |(j,k^{s})_s}f(d)g \left(\frac{k}{d}\right)  
 = \sum_{\substack{d|k  \\d^{s}|j}}f(d)g\left(\frac{k}{d}\right).  
\end{equation}
In \cite{K2}, the first author studied the partial sum of the weighted average of $s_{k}^{(s)}(j)$ and proved that  
\begin{multline*}
 \sum_{k\leq x}  \frac{1}{k^{s(r+1)}} \sum_{j=1}^{k^s} j^{r} s_{k}^{(s)}(j)    
  = \frac{1}{2}\sum_{dl\leq x}\frac{f(d)}{d^s}\frac{g(l)}{l^s}  \\
  + \frac{1}{r+1}\sum_{dl\leq x}\frac{f(d)}{d^s} g(l) 
    + \frac{1}{r+1}\sum_{m=1}^{[r/2]}\binom{r+1}{2m} B_{2m}\sum_{dl\leq x}\frac{f(d)}{d^s}\frac{g(l)}{l^{2ms}}.    
\end{multline*}
Now  we take $f*\mu$ in place of $f$ and $g={\bf 1}$ into the above   
 and use the identity  $(f*\mu)*\1=f$  to obtain 
 \begin{align}                                                                                                     \label{kiuchi}
 M_{r}^{(s)}(x;f) 
& := \sum_{k\leq x}  \frac{1}{k^{s(r+1)}} \sum_{j=1}^{k^s} j^{r} \sum_{\substack{d|k \\ d^{s}|j}}(f*\mu)(d)      \\
&  = \frac{1}{2} \sum_{n\leq x} \frac{f(n)}{n^s}   + \frac{1}{r+1}\sum_{dl\leq x}\frac{(f*\mu)(d)}{d^s}  \nonumber  \\
& \qquad  \qquad    
+ \frac{1}{r+1} \sum_{m=1}^{[r/2]} \binom{r+1}{2m} B_{2m} \sum_{dl\leq x} \frac{(f*\mu)(d)}{d^s} \frac{1}{l^{2ms}}.       \nonumber
\end{align} 
In this paper, we provide two applications of the identity  \eqref{kiuchi}.
First  one is in the case when $f={\rm id}_{s+a}$. We have   
\begin{thm}  \label{th5}
For any  real number $x >1$, fixed  positive integers $r$, $s\geq 2$,  and fixed  number $a$ such that $- 1 < a < 0$,   we have  
\begin{multline}                                                          \label{sum_th5}
  M_{r}^{(s)}(x;{\rm id}_{s+a}) 
=\frac{\zeta(1-a)}{(r+1)\zeta(s+1)}x  + \frac{1}{2(1+a)}x^{1+a}  \\
 + \frac{x^{1+a} }{(r+1)(1+a)\zeta(s+a+1)}\left(\zeta(1+a) +  \sum_{m=1}^{[r/2]}\binom{r+1}{2m} B_{2m}\zeta(2ms+a+1)\right) \\
+  K_{r}^{(s)}(x;\id_{s+a}), 
\end{multline}
where  
\begin{multline}                                                             \label{K-S-th5}
 K_{r}^{(s)}(x;{\rm id}_{s+a})  
= \frac{1}{r+1} \sum_{n\leq x}\frac{\mu(n)}{n^s}\Delta_{a}\left(\frac{x}{n}\right) \\+ 
   \frac{1}{r+1}\sum_{m=1}^{[r/2]}\binom{r+1}{2m}B_{2m} \frac{\zeta(-a)\zeta(2ms)}{\zeta(s)} 
 + \frac{\zeta(-a)}{2}\\ - \frac{\zeta(-a)}{2(r+1)\zeta(s)} + \Ocal_{r,s,a}\left(x^{a}\right).       
\end{multline}
\end{thm}
Using  (\ref{Peter}),  the error  term  $K_{r}^{(s)}(x;{\rm id}_{s+a})$    is estimated by 
$\Ocal_{r,s}\left(x^{\frac{1+a}{3}+\varepsilon}\right)$.  Then   we get   
$$
\lim_{x\to\infty}\frac{ M_{r}^{(s)}(x;{\rm id}_{s+a}) }{x}    =  \frac{\zeta(1-a)}{(r+1)\zeta(s+1)}.   
$$
The second  application of the formula \eqref{kiuchi}  is given in the next subsection.  


\subsection{Evaluations of  $M_{r}^{(s)}(x;h*{\rm id}_{s}) $ and  $M_{r}^{(s)}(x;h*{\rm id}_{s+a}) $}

Let $H_{h}(\alpha)$  be the Dirichlet series  defined by 
$$
H_{h}(\alpha)=\sum_{n=1}^{\infty}\frac{h(n)}{n^{\alpha}},
$$
which  is   absolutely convergent  in the half-plane  ${{\rm Re}~\alpha} > \sigma_{c}$. 
Then the first  derivative $H^{'}_{h}(\alpha)$ of $H_{h}(\alpha)$  is represented, in the same half-plane, by the Dirichlet series  
$$
H_{h}^{'}(\alpha) = - \sum_{n=1}^{\infty}\frac{h(n)\log n}{n^\alpha}. 
$$
Taking  $f=h*\id_{s}$ and then $f=h* \id_{s+a}$,  for any arithmetical function  $h$, into  (\ref{kiuchi}), we see that    
\begin{multline}                                                             \label{Sumaia-K}
M_{r}^{(s)}(x;h*\id_{s})   
 = \frac{1}{2}\sum_{dl\leq x}\frac{h(d)}{d^s} 
+ \frac{1}{r+1}\sum_{dl\leq x} \frac{(h*\mu)(d)}{d^s} \tau(l)  \\
+ \frac{1}{r+1}\sum_{m=1}^{[r/2]}\binom{r+1}{2m} B_{2m}\sum_{dl\leq x}\frac{(h*\phi_{s})(d)}{d^s}\frac{1}{l^{2ms}}, 
\end{multline}
and         
\begin{multline}                                                         
\label{Sumaia-K1}
 M_{r}^{(s)}(x;h*\id_{s+a})   
 = \frac{1}{2}\sum_{dl\leq x}\frac{h(d)}{d^s} l^a
+ \frac{1}{r+1}\sum_{dl\leq x} \frac{(h*\mu)(d)}{d^s} \sigma_{a}(l)  \\
+ \frac{1}{r+1}\sum_{m=1}^{[r/2]}\binom{r+1}{2m} B_{2m}\sum_{dl\leq x}\frac{(h*\phi_{s+a})(d)}{d^s}\frac{1}{l^{2ms}}. 
\end{multline}
Under certain  conditions of the arithmetical function  $h$,  we  provide the following results.   
\begin{thm}  \label{th71}
Assume that $h(n)$ is estimated by $O\left(n^{\varepsilon}\right)$ 
for any small  number $\varepsilon >0$.  
For any  real number $x >1$,  fixed positive  integers $r$ and $s \geq 2$,    we have
\begin{multline}                                                              \label{sum_th770}
 M_{r}^{(s)}(x;h*\id_{s})   
 = \frac{H_{h}(s+1)}{(r+1)\zeta(s+1)}x\log x  
  + \frac{H_{h}(s+1)}{(r+1)\zeta(s+1)} \times \\
\times \left(2\gamma -1 -  \frac{\zeta'(s+1)}{\zeta(s+1)} + \frac{H^{'}_{h}(s+1)}{H_{h}(s+1)} 
+ \sum_{m=1}^{[r/2]}\binom{r+1}{2m} B_{2m}\zeta(2ms+1)\right)x  \\
 + \frac{H_{h}(s+1)}{2} x   +  K_{r}^{(s)}(x;h*\id_{s}),  
\end{multline}
where   
\begin{multline}                                                              \label{K-S-th771}
K_{r}^{(s)}(x;h*\id_{s}) 
= \frac{1}{r+1} \sum_{n\leq x}\frac{(h*\mu)(n)}{n^s}\Delta_{}\left(\frac{x}{n}\right)
 -  \frac{H_{h}(s)}{2}  \\
- \frac{H_{h}(s)}{2(r+1)\zeta(s)}\sum_{m=1}^{[r/2]}\binom{r+1}{2m}B_{2m}\zeta(2ms) 
- \sum_{d\leq x}\frac{h(d)}{d^s}\vartheta\left(\frac{x}{d}\right) + \Ocal_{r,s}\left(x^{1-s+\varepsilon}\right). 
\end{multline}
\end{thm}
\begin{thm}  \label{th72}
Assume that $h(n)$ is estimated by $O\left(n^{\varepsilon}\right)$ 
for any small  number $\varepsilon >0$.  
For any  real number $x >1$,   fixed positive  integers $r$ and $s \geq 2$,    we have
\begin{multline}                                                           \label{sum_th777}
 M_{r}^{(s)}(x;h*\id_{s+a})   
 = \frac{\zeta(1-a)H_{h}(s+1)}{(r+1)\zeta(s+1)}x + \frac{H_{h}(s+a+1)}{(r+1)(1+a)\zeta(s+a+1)} \times  \\
 \times \left(\zeta(1+a) +  \sum_{m=1}^{[r/2]}\binom{r+1}{2m} B_{2m}\zeta(2ms+a+1)\right)x^{1+a}  \\
 + \frac{H_{h}(s+a+1)}{2(1+a)}x^{1+a} +  K_{r}^{(s)}(x;h*\id_{s+a}),
\end{multline}
with any fixed  number $a$ such that $-1<a<0$,   
where  
\begin{multline}                                                             \label{K-S-th777}
  K_{r}^{(s)}(x;h*\id_{s+a}) 
= \frac{1}{r+1} \sum_{n\leq x}\frac{(h*\mu)(n)}{n^s}\Delta_{a}\left(\frac{x}{n}\right) + \frac{\zeta(-a)H_{h}(s)}{2}  \\ 
+ \frac{\zeta(-a)H_{h}(s)}{(r+1)\zeta(s)}\left(\sum_{m=1}^{[r/2]}\binom{r+1}{2m}B_{2m}\zeta(2ms) -\frac12\right) 
  + \Ocal_{s,r,a}\left(x^{a}\right).  
\end{multline}
\end{thm}
Many interesting applications of Theorems \ref{th71} and \ref{th72} are given in Section~\ref{section2}.  
\section{Applications of Theorems \ref{th71} and \ref{th72}}
\label{section2}
In this section, we give applications of Theorems \ref{th71} and \ref{th72} for various multiplicative functions 
such as $\mu$, $\tau$, $\phi_{s}$, $\psi_{s}$, ${\rm id}_{s}$  and others. 
Define the functions $D_{s}(x)$ and  ${\widetilde{D_{s}}}(x)$   by       
\begin{equation}                                                               \label{lem411}
D_{s}(x) = -\sum_{d\leq x}\frac{\mu(d)}{d^s}\vartheta\left(\frac{x}{d}\right) - \frac{1}{2\zeta(s)}
\end{equation}
and
\begin{equation}                                                          \label{lem421}
{\widetilde{D_{s}}}(x) = -\sum_{d\leq x}\frac{|\mu(d)|}{d^s}\vartheta\left(\frac{x}{d}\right) - \frac{\zeta(s)}{2\zeta(2s)}, 
\end{equation}
where  
$
\vartheta(x)=x - [x] - \frac{1}{2}.   
$
First, we take $h=\mu$,  $|\mu|$ and $h=\tau$ into (\ref{sum_th770}) and  use the identities  
\begin{equation}                                                             \label{HHH}
H_{\mu}(s+1)=\frac{1}{\zeta(s+1)}, \ \   
H_{\mu}^{'}(s+1)=-\frac{\zeta'(s+1)}{\zeta(s+1)^2},\ \   
H_{|\mu|}(s+1)= \frac{\zeta(s+1)}{\zeta(2s+2)},
\end{equation}
\begin{equation}                                                        \label{HHH1}
H_{|\mu|}^{'}(s+1)= \frac{\zeta'(s+1)\zeta(2s+2)-2\zeta'(2s+2)\zeta(s+1)}{\zeta(2s+2)^2}, 
\end{equation} 
\begin{equation}                                                          \label{HHH2}
H_{\tau}(s+1) = \zeta(s+1)^2, 
\quad \text{and} \quad  
H^{'}_{\tau}(s+1) = 2\zeta'(s+1)\zeta(s+1)
\end{equation} 
to deduce the following results:  
\begin{cor}
For  any  real number $x >1$  and   fixed  positive  integers $r$ and $s \geq 2$,  we have 
\begin{multline}                                                             \label{sum_th6}
  M_{r}^{(s)}(x;\phi_{s}) 
=\frac{1}{(r+1)\zeta(s+1)^2}x\log x + \frac{1}{2\zeta(s+1)} x  \\
 + \frac{x}{(r+1)\zeta(s+1)^2}
\left(2\gamma -1 - 2 \frac{\zeta'(s+1)}{\zeta(s+1)} + \sum_{m=1}^{[r/2]}\binom{r+1}{2m} B_{2m}\zeta(2ms+1)\right)  \\
 +  K_{r}^{(s)}(x;\phi_{s}),  
\end{multline}
\begin{multline}                                                          \label{sum_th7}
M_{r}^{(s)}(x;\psi_{s})   
 = \frac{1}{(r+1)\zeta(2s+2)}x\log x + \frac{\zeta(s+1)}{2\zeta(2s+2)} x  \\
 + \frac{x}{(r+1)\zeta(2s+2)}
\left(2\gamma -1 - 2 \frac{\zeta'(2s+2)}{\zeta(2s+2)} + \sum_{m=1}^{[r/2]}\binom{r+1}{2m} B_{2m}\zeta(2ms+1)\right)  \\
  +  K_{r}^{(s)}(x;\psi_{s}), 
\end{multline}
and 
\begin{multline}                                                           \label{sum_th77077}
  M_{r}^{(s)}(x;\tau*\id_{s})   
 = \frac{\zeta(s+1)}{r+1}x\log x + \frac{\zeta(s+1)^2}{2} x  \\
 + \frac{\zeta(s+1)}{r+1}
\left(2\gamma -1  + \frac{\zeta'(s+1)}{\zeta(s+1)} 
+ \sum_{m=1}^{[r/2]}\binom{r+1}{2m} B_{2m}\zeta(2ms+1)\right)x   \\
+  K_{r}^{(s)}(x;\tau*\id_{s}),           
\end{multline}
where  
\begin{multline}                                                             \label{K-S-th6}
  K_{r}^{(s)}(x;\phi_{s}) 
= \frac{1}{r+1} \sum_{n\leq x}\frac{(\mu*\mu)(n)}{n^s}\Delta_{}\left(\frac{x}{n}\right) + \frac{D_{s}(x)}{2}   \\
 - \frac{1}{2(r+1)\zeta(s)^2}\sum_{m=1}^{[r/2]}\binom{r+1}{2m}B_{2m}\zeta(2ms) + \Ocal_{r,s}\left(x^{1-s}(\log x)^2\right),  
\end{multline}
\begin{multline}                                                          \label{K-S-th71}
  K_{r}^{(s)}(x;\psi_{s}) 
= \frac{1}{r+1} \sum_{n\leq x}\frac{(\mu*|\mu|)(n)}{n^s}\Delta_{}\left(\frac{x}{n}\right) + \frac{{\widetilde{D_{s}}}(x)}{2}   \\
 - \frac{1}{2(r+1)\zeta(2s)}\sum_{m=1}^{[r/2]}\binom{r+1}{2m}B_{2m}\zeta(2ms) + \Ocal_{r,s}\left(x^{1-s}(\log x)^2\right),  
\end{multline}
and 
\begin{multline}                                                          \label{K-S-th77177}
 K_{r}^{(s)}(x;\tau*\id_{s}) 
= \frac{1}{r+1} \sum_{n\leq x}\frac{(\tau*\mu)(n)}{n^s}\Delta_{}\left(\frac{x}{n}\right)
 -  \frac{\zeta(s)^2}{2}  \\
 - \frac{\zeta(s)}{2(r+1)}\sum_{m=1}^{[r/2]}\binom{r+1}{2m}B_{2m}\zeta(2ms) 
- \sum_{d\leq x}\frac{\tau(d)}{d^s}\vartheta\left(\frac{x}{d}\right) + \Ocal_{r,s}\left(x^{1-s+\varepsilon}\right). 
\end{multline}
\end{cor}
Second,  by taking   $h=\mu$,  $|\mu|$ and $\tau$ into \eqref{sum_th777} and using   the identities 
 (\ref{HHH}),  (\ref{HHH1})  and  (\ref{HHH2}),  we immediately get the  following formulas:  
\begin{cor}  
 Let $a$ be  any fixed  number with $-1<a<0$. 
For  any  real number $x >1$,   fixed  positive  integers $r$ and $s \geq 2$,   we have
\begin{multline}                                                          \label{sum_th61}
 M_{r}^{(s)}(x;\phi_{s+a}) 
=\frac{\zeta(1-a)}{(r+1)\zeta(s+1)^2}x + \frac{x^{1+a}}{2(1+a)\zeta(s+a+1)}  \\
+ \frac{x^{1+a}}{(r+1)(1+a)\zeta(s+a+1)^2}\left(\zeta(1+a) +  \sum_{m=1}^{[r/2]}\binom{r+1}{2m} B_{2m}\zeta(2ms+a+1)\right)   \\
 +  K_{r}^{(s)}(x;\phi_{s+a}),                   
\end{multline}
\begin{multline}                                                             \label{sum_th77}
  M_{r}^{(s)}(x;\psi_{s+a})   
 = \frac{\zeta(1-a)}{(r+1)\zeta(2s+2)}x + \frac{\zeta(s+a+1)}{2(1+a)\zeta(2s+2a+2)}x^{1+a}  \\
 + \frac{x^{1+a}}{(r+1)(1+a)\zeta(2s+2a+2)}\left(\zeta(1+a) +  \sum_{m=1}^{[r/2]}\binom{r+1}{2m} B_{2m}\zeta(2ms+a+1)\right)  \\
 +  K_{r}^{(s)}(x;\psi_{s+a}),           
\end{multline}
and 
\begin{multline}                                                           \label{sum_th77777}
 M_{r}^{(s)}(x;\tau*id_{s+a})   
 = \frac{\zeta(1-a)\zeta(s+1)}{r+1}x + \frac{\zeta(s+a+1)^2}{2(1+a)}x^{1+a}  \\
+ \frac{\zeta(s+a+1)}{(r+1)(1+a)}\left(\zeta(1+a) +  \sum_{m=1}^{[r/2]}\binom{r+1}{2m} B_{2m}\zeta(2ms+a+1)\right)x^{1+a} \\
 +  K_{r}^{(s)}(x;\tau*{\id_{s+a}}),       
\end{multline}
where  
\begin{multline}                                                          \label{K-S-th61}
  K_{r}^{(s)}(x;\phi_{s+a}) 
= \frac{1}{r+1} \sum_{n\leq x}\frac{(\mu*\mu)(n)}{n^s}\Delta_{a}\left(\frac{x}{n}\right) + \frac{\zeta(-a)}{2\zeta(s)}  \\ 
+ \frac{\zeta(-a)}{(r+1)\zeta(s)^2} \left(\sum_{m=1}^{[r/2]}\binom{r+1}{2m}B_{2m}\zeta(2ms)  -\frac12 \right)
  + \Ocal_{s,r,a}\left(x^{a}\right), 
\end{multline}
\begin{multline}                                                          \label{K-S-th77}
  K_{r}^{(s)}(x;\psi_{s+a}) 
= \frac{1}{r+1} \sum_{n\leq x}\frac{(\mu*|\mu|)(n)}{n^s}\Delta_{a}\left(\frac{x}{n}\right) + \frac{\zeta(-a)\zeta(s)}{2\zeta(2s)}  \\ 
+ \frac{\zeta(-a)}{(r+1)\zeta(2s)}\left(\sum_{m=1}^{[r/2]}\binom{r+1}{2m}B_{2m}\zeta(2ms) -\frac{1}{2}\right) 
  + \Ocal_{s,r,a}\left(x^{a}\right), 
\end{multline}
and 
\begin{multline}                                                          \label{K-S-th77777}
  K_{r}^{(s)}(x;\tau*\id_{s+a}) 
= \frac{1}{r+1} \sum_{n\leq x}\frac{(\tau*\mu)(n)}{n^s}\Delta_{a}\left(\frac{x}{n}\right) + \frac{\zeta(-a)\zeta(s)^2}{2}  \\ 
+ \frac{\zeta(-a)\zeta(s)}{r+1}\left(\sum_{m=1}^{[r/2]}\binom{r+1}{2m}B_{2m}\zeta(2ms) -\frac{1}{2}\right) 
  + \Ocal_{s,r,a}\left(x^{a}\right). 
\end{multline}
\end{cor}
The formulas \eqref{sum_th6} and \eqref{sum_th61} give us an analogue of \eqref{K-phi} and \eqref{gcd331}, respectively. 
Using \eqref{Peter},  the formulas~\eqref{K-S-th6} and \eqref{K-S-th61} are estimated by 
$\Ocal_{r,s}\left(x^{\frac{1}{3}+\varepsilon}\right)$ and  $\Ocal_{r,s,a}\left(x^{\frac{1+a}{3}+\varepsilon}\right)$,  respectively. 
Then, we  have the relations 
$$
\lim_{x\to\infty}\frac{ M_{r}^{(s)}(x;\phi_{s})}{x\log x}      =  \frac{1}{(r+1)\zeta(s+1)^2},     
\ \   {\rm and} \  \      
\lim_{x\to\infty}\frac{ M_{r}^{(s)}(x;\phi_{s+a}) }{x}     =  \frac{\zeta(1-a)}{(r+1)\zeta(s+1)^2}.   
$$
Similarly, the formulas \eqref{K-S-th71} and \eqref{K-S-th77} are   estimated by 
$\Ocal_{r,s}\left(x^{\frac{1}{3}+\varepsilon}\right)$ and  $\Ocal_{r,s,a}\left(x^{\frac{1+a}{3}+\varepsilon}\right)$, respectively.
Then, we have 
$$
\lim_{x\to\infty}\frac{ M_{r}^{(s)}(x;\psi_{s})}{x\log x}      =  \frac{1}{(r+1)\zeta(2s+2)},       
\ \ {\rm and} \ \      
\lim_{x\to\infty}\frac{ M_{r}^{(s)}(x;\psi_{s+a})}{x}      =  \frac{\zeta(1-a)}{(r+1)\zeta(2s+2)}.   
$$
We use  the formulas \eqref{sum_th6} and \eqref{sum_th7} to deduce 
\begin{multline*}
\zeta(s+1)^{2}M_{r}^{(s)}(x;\phi_s) - \zeta(2s+2)M_{r}^{(s)}(x;\psi_s) 
 = \\\frac{2}{r+1}\left(\frac{\zeta'(2s+2)}{\zeta(2s+2)}-\frac{\zeta'(s+1)}{\zeta(s+1)}\right)x  
 +  \zeta^{}(s+1)^{2} K_{r}^{(s)}(x;\phi_s) -  \zeta^{}(2s+2) K_{r}^{(s)}(x;\psi_s).  
\end{multline*}
Furthermore, using \eqref{sum_th61} and \eqref{sum_th77}  we obtain that 
\begin{multline*}
\zeta(s+a+1)^{2}M_{r}^{(s)}(x;\phi_{s+a}) - \zeta(2s+2a+2)M_{r}^{(s)}(x;\psi_{s+a}) \\
 = \frac{\zeta(1-a)}{r+1}\left(\frac{\zeta(s+a+1)^2}{\zeta(s+1)^2} - \frac{\zeta(2s+2a+2)}{\zeta(2s+2)}\right)x \\ 
+  \zeta^{}(s+a+1)^{2} K_{r}^{(s)}(x;\phi_{s+a}) -  \zeta(2s+2a+2) K_{r}^{(s)}(x;\psi_{s+a}). 
\end{multline*}
This latter is an  analogue of the relation~\eqref{Remark_th2}.
\begin{rem}
Let the arithmetical function $h=\xi_{q}$ be $q$-free number for any fixed positive integer $q$ defined  by 
\begin{equation}                                                            \label{LKJ}
H_{\xi_q}(s) = \sum_{n=1}^{\infty}\frac{\xi_{q}(n)}{n^s} = \frac{\zeta(s)}{\zeta(qs)}.  
\end{equation}
Then the first derivative of $H_{\xi_q}(s)$ with respect to $s$ is  
\begin{equation}                                                                  \label{HGF}
H'_{\xi_q}(s)=-\sum_{n=1}^{\infty}\frac{\xi_{q}(n)\log n}{n^s} = \frac{\zeta'(s)\zeta(qs)-q\zeta'(qs)\zeta(s)}{\zeta(qs)^2}.  
\end{equation}
Similarly  as  above,  we can take  $h=\xi_{q}$ into \eqref{sum_th770} and \eqref{sum_th777}, and  use \eqref{LKJ} and \eqref{HGF} to  derive  asymptotic formulas of 
$
M_{r}^{(s)}(x;\xi_{q}*\id_{s})  
$
and 
$
M_{r}^{(s)}(x;\xi_{q}*\id_{s+a}),   
$
respectively.  
\end{rem}
\section{Some Lemmas}
Before we proceed with the proof of the main results, 
we need to  give  some  auxiliary  lemmas. 
\begin{lem}    
\label{lem1}
For any sufficiently large  number $x>1$ and  fixed  number $a$  such  that $-1<a<0$, 
 we have
\begin{equation}    
\label{lem1-apo} 
 \sum_{n\leq x}n^{a} = \frac{x^{1+a}}{1+a} + \zeta(-a) + \Ocal_{a}\left(x^{a}\right),                       
 \end{equation} 
\begin{equation} 
\label{lem1-mu}  
 \sum_{n\leq x}\frac{\mu(n)}{n} = \Ocal\left(\delta(x)\right),      
 \end{equation}
 \begin{equation}
  \label{lem1-mu1}
 \sum_{n\leq x}\frac{\mu(n)}{n^2} = \frac{1}{\zeta(2)} + \Ocal\left(\frac{\delta(x)}{x}\right),     
\end{equation}
and 
\begin{equation}
\label{lem1-mu2} 
 \sum_{n\leq x}\frac{\mu(n)}{n^{2+a}}
= \frac{1}{\zeta(2+a)}  + \Ocal_{a}\left(\frac{\delta(x)}{x^{1+a}}\right),
\end{equation}
where   
\begin{equation}                                                           \label{deltaas} 
\delta(x) = {\rm exp}\left(-C\frac{(\log x)^{3/5}}{(\log\log x)^{1/5}}\right)
\end{equation}  
with $C$ being a positive constant.
\end{lem}
\begin{proof}
 The formula~\eqref{lem1-apo} follows from  Theorem 3.2 (b) in~\cite{Ap}. 
The formulas \eqref{lem1-mu1} and \eqref{lem1-mu2} follow from Lemmas 2.2 in~\cite{SS}.  
Furthermore, the proof of the formula \eqref{lem1-mu} can be found in~\cite{Ji}.  
\end{proof}
\begin{lem}
\label{lem2}
For any sufficiently large  number $x>1$,   fixed  number $a$ such that $-1<a<0$  and 
fixed integer $m \geq 2$,   we have 
\begin{multline}                                                            \label{lem2-sigma1}                                                    
   \sum_{n\leq x} \left(\frac{\mu}{{\rm id}}*\sigma_{a}\right)(n)  
= \frac{\zeta(1-a)}{\zeta(2)}x                                            
 + \frac{\zeta(1+a)}{(1+a)\zeta(2+a)}x^{1+a} \\
  + \sum_{d\leq x}\frac{\mu(d)}{d}\Delta_{a}\left(\frac{x}{d}\right) + \Ocal_{a}\left(\delta(x)\right),   
\end{multline}
where $\delta(x)$ is given by \eqref{deltaas},     
and 
\begin{equation}   
\label{lem2-la}
\sum_{dl\leq x} \left(\frac{\mu}{{\rm id}}*\frac{1}{{\rm id}_{2m}}\right)(d)l^{a} =
\frac{\zeta(2m+a+1)}{(1+a)\zeta(2+a)} x^{1+a} + \Ocal_{a,m}\left(1\right).
\end{equation}
\end{lem}

\begin{proof}  
From \eqref{Delta-a}, we find that   
\begin{eqnarray*}
 \sum_{n\leq x}\left(\frac{\mu}{{\rm id}}*\sigma_{a}\right)(n)  
&=&  \sum_{d\leq x}\frac{\mu(d)}{d} \sum_{l\leq x/d}\sigma_{a}(l)  
\\&=& \sum_{d\leq x}\frac{\mu(d)}{d}\left(\zeta(1-a)\frac{x}{d} + \frac{\zeta(1+a)}{1+a}\left(\frac{x}{d}\right)^{1+a}
 - \frac{\zeta(-a)}{2} + \Delta_{a}\left(\frac{x}{d}\right)\right) 
\\&=& \zeta(1-a) x\sum_{d\leq x}\frac{\mu(d)}{d^2} + \frac{\zeta(1+a)}{1+a} x^{1+a} \sum_{d\leq x}\frac{\mu(d)}{d^{2+a}} 
\\&& \qquad \qquad  \qquad \qquad -  \frac{\zeta(-a)}{2} \sum_{d\leq x}\frac{\mu(d)}{d} + \sum_{d\leq x}\frac{\mu(d)}{d}\Delta_{a}\left(\frac{x}{d}\right). 
\end{eqnarray*}
Using  \eqref{lem1-mu}, \eqref{lem1-mu1} and \eqref{lem1-mu2},  we get   
\begin{multline*}
\sum_{n\leq x}\left(\frac{\mu}{{\rm id}}*\sigma_{a}\right)(n)  
= \zeta(1-a) x\left(\frac{1}{\zeta(2)} + O\left(\frac{\delta(x)}{x}\right)\right) \\
+  \frac{\zeta(1+a)}{1+a} x^{1+a} \left(\frac{1}{\zeta(2+a)} + \Ocal\left(\frac{\delta(x)}{x^{1+a}}\right)\right) 
 + \sum_{d\leq x}\frac{\mu(d)}{d}\Delta_{a}\left(\frac{x}{d}\right)  + \Ocal_{a}\left(\delta(x)\right), 
\end{multline*}
then \eqref{lem2-sigma1} is proved.  
Similarly,  using \eqref{lem1-apo} and \eqref{lem1-mu} we find  that  
\begin{align*}
&  \sum_{dl\leq x}\left(\frac{\mu}{{\rm id}}*\frac{1}{{\rm id}_{2m}}\right)(d)l^{a}  
=  \sum_{d\leq x}\left(\frac{\mu}{{\rm id}}*\frac{1}{{\rm id}_{2m}}\right)(d) \sum_{l\leq x/d} l^{a}  \\
&=  \frac{x^{1+a}}{1+a} \sum_{dl\leq x}\frac{\mu(d)}{d^{2+a}}\frac{1}{l^{2m+a+1}} 
+   \zeta(-a)  \sum_{dl\leq x}\frac{\mu(d)}{d}\frac{1}{l^{2m}}  + \Ocal_{a}\left(x^{a} \sum_{dl\leq x}\frac{1}{d^{1+a}}\frac{1}{l^{2m+a}}\right)  \\
&= \frac{\zeta(2m+a+1)}{(1+a)\zeta(2+a)} x^{1+a} + \Ocal_{a,m}\left(1\right).    
\end{align*}
This completes the proof of \eqref{lem2-la}.    
\end{proof}
\begin{lem}
\label{lem3}
For any sufficiently large  number $x>1$, fixed  number $a$  such that $-1<a<0$,  and 
fixed  integer $m \geq 2$,  we have 
\begin{multline}                                                          
\label{lem2-mu}
\sum_{dl\leq x} \frac{(\phi_{1+a}*\mu)(d)}{d}  
= \frac{\zeta(1-a)}{\zeta(2)^2} x  + \frac{\zeta(1+a)}{(1+a)\zeta(2+a)^2}x^{1+a}                        
\\ + \sum_{d\leq x}\frac{(\mu*\mu)(d)}{d}\Delta_{a}\left(\frac{x}{d}\right) + \Ocal_{a}\left((\log x)^2\right),    
\end{multline}
\begin{multline}
\label{lem2-psi} 
 \sum_{dl\leq x} \frac{(\psi_{1+a}*\mu)(d)}{d}  
= \frac{\zeta(1-a)}{\zeta(4)} x  + \frac{\zeta(1+a)}{(1+a)\zeta(4+2a)}x^{1+a}                                       \\
+ \sum_{d\leq x}\frac{(|\mu|*\mu)(d)}{d}\Delta_{a}\left(\frac{x}{d}\right) + \Ocal_{a}\left((\log x)^2\right),  
\end{multline}
\begin{equation}
\label{lem2-l2m} 
\sum_{dl\leq x} \frac{(\phi_{1+a}*\mu)(d)}{d}  \frac{1}{l^{2m}} =
\frac{\zeta(a+2m+1)}{(a+1)\zeta(a+2)^2}x^{a+1} + \Ocal_{a,m}\left((\log x)^2\right),                                            
\end{equation}
and 
\begin{equation}
\label{lem2-psi2m}
\sum_{dl\leq x} \frac{(\psi_{1+a}*\mu)(d)}{d}  \frac{1}{l^{2m}} =
\frac{\zeta(2m+a+1)}{(1+a)\zeta(4+2a)} x^{1+a} + \Ocal_{a,m}\left((\log x)^2\right).                                            
\end{equation}
\end{lem}
\begin{proof}  
First we  are  going to prove \eqref{lem2-mu} and \eqref{lem2-l2m}.
Notice that 
$$
\left(\frac{\mu}{\id}*\frac{\phi_{1+a}}{\id}\right)*{\1} = 
\frac{\mu}{\id}*\id_{a}*\frac{\mu}{\id}*\1 = 
\frac{\mu*\mu}{\id}*\sigma_{a}. 
$$
Then, we write   
$$
\sum_{dl\leq x}\frac{(\phi_{1+a}*\mu)(d)}{d}=\sum_{d\leq x}\frac{(\mu*\mu)(d)}{d}\sum_{l\leq x/d}\sigma_{a}(l).
$$
Using  (\ref{Delta-a}), the formulas 
$$
\sum_{d\leq x}\frac{(\mu*\mu)(d)}{d^2} = \frac{1}{\zeta(2)^2} + \Ocal\left(\frac{\log x}{x}\right), \ 
\sum_{d\leq x}\frac{(\mu*\mu)(d)}{d^{2+a}} = \frac{1}{\zeta(2+a)^2} + \Ocal_{a}\left(\frac{\log x}{x^{1+a}}\right),
$$
and the estimate 
$
\sum_{d\leq x}\frac{(\mu*\mu)(d)}{d} \ll \sum_{d\leq x}\frac{\tau(d)}{d}  \ll (\log x)^2,  
$
we find that 
\begin{eqnarray*}
\sum_{dl\leq x} \frac{(\phi_{1+a}*\mu)(d)}{d} 
&=&  \sum_{d\leq x}  \frac{(\mu*\mu)(d)}{d} \sum_{l\leq x/d} \sigma_{a}(l)   \\
&=& \zeta(1-a)x \sum_{d\leq x}\frac{(\mu*\mu)(d)}{d^{2}} 
+  \frac{\zeta(1+a)}{1+a} x^{1+a} \sum_{d\leq x}\frac{(\mu*\mu)(d)}{d^{2+a}}  \\
& & \qquad \qquad -  \frac{\zeta(-a)}{2} \sum_{d\leq x}\frac{(\mu*\mu)(d)}{d} 
+  \sum_{d\leq x}\frac{(\mu*\mu)(d)}{d}\Delta_{a}\left(\frac{x}{d}\right).
\end{eqnarray*}
This completes the proof of \eqref{lem2-mu}.  
Since 
$$
\frac{\phi_{1+a}*\mu}{\id}*\frac{1}{\id_{2m}} = \frac{\mu*\mu}{\id}*\frac{\sigma_{a+2m}}{\id_{2m}}, 
$$
then, using the  formulas  
$$
\sum_{l\leq x}\frac{\sigma_{a+2m}(l)}{l^{2m}} = 
\frac{\zeta(a+2m+1)}{a+1}x^{a+1} + \zeta(-a)\zeta(2m) + \Ocal_{a,m}\left(x^{a}\right)
$$
and \eqref{lem1-mu2}, we get  
\begin{align*}
& \sum_{dl\leq x}\frac{(\mu*\mu)(d)}{d}\frac{1}{l^{2m}} = 
  \sum_{d\leq x}\frac{(\mu*\mu)(d)}{d}\sum_{l\leq x/d}\frac{\sigma_{a+2m}}{l^{2m}} \\
&=\sum_{d\leq x}\frac{(\mu*\mu)(d)}{d}\left(\frac{\zeta(a+2m+1)}{a+1}\left(\frac{x}{d}\right)^{a+1}
 + \zeta(-a)\zeta(2m) + \Ocal\left(\frac{x^a}{d^a}\right)\right) \\
&=\frac{\zeta(a+2m+1)}{(a+1)\zeta(a+2)^2}x^{a+1} + \Ocal_{a,m}\left((\log x)^2\right).  
\end{align*}
This completes the proof of \eqref{lem2-l2m}.
Similarly, we use the facts  
$$
\left(\frac{\mu}{\id}*\frac{\psi_{1+a}}{\id}\right)*{\1}  = \frac{|\mu|*\mu}{\id}*\sigma_{a},  
\ \  
\frac{\psi_{1+a}*\mu}{{\rm id}}*\frac{1}{\id_{2m}} = \frac{|\mu|*\mu}{{\rm id}}*\frac{\sigma_{a+2m}}{\id_{2m}}, 
$$
$$
\sum_{d\leq x}\frac{(|\mu|*\mu)(d)}{d^2} = \frac{1}{\zeta(4)} + \Ocal\left(\frac{\log x}{x}\right), \   
\sum_{d\leq x}\frac{(|\mu|*\mu)(d)}{d^{2+a}} = \frac{1}{\zeta(4+2a)} + \Ocal_{a}\left(\frac{\log x}{x^{1+a}}\right),
$$
and the estimate 
$
\sum_{d\leq x}\frac{(|\mu|*\mu)(d)}{d} \ll (\log x)^2  
$
to prove \eqref{lem2-psi} and \eqref{lem2-psi2m}.  
\end{proof}
In order to prove  Theorems~\ref{th3} and \ref{th4} we need the following lemmas:    
\begin{lem}     
\label{lem30}
For  $1\ll N\ll x$   and $-1<a<0$,  we have 
\begin{equation}                                                             \label{lem3011}
\Delta_{a}(x) = 
\frac{x^{\frac14 +\frac{a}{2}}}{\pi\sqrt{2}}
\sum_{n\leq N}\frac{\sigma_{a}(n)}{n^{\frac{3}{4}+\frac{a}{2}}}\cos\left(4\pi\sqrt{nx}-\frac{\pi}{4}\right) 
+ \Ocal\left(x^{\frac12 + \varepsilon}N^{-\frac12}\right).    
\end{equation}
\end{lem}
\begin{proof}
The formula \eqref{lem3011} of  Vorono\"{i}'s  type follows from a special case of Theorem~1 in~\cite{k}. 
\end{proof}
\begin{lem}     
\label{lem35}
For  any sufficiently large number $x>1$  
and  fixed number $a$ such that $-1 < a <0$, we have 
\begin{equation}                                                          \label{lem3055}
\sum_{n\leq x}\sigma_{a}(n)^{2} = 
\frac{\zeta^{2}(1-a)\zeta(1-2a)}{\zeta(2-2a)}  x  +  O\left(x^{1+\frac{a}{4}}(\log x)^{2}\right).    
\end{equation}
\end{lem}
\begin{proof}
The formula \eqref{lem3055} follows from  Lemma~1 in~\cite{KT1}. 
\end{proof}
\begin{lem}       
\label{lem40}
Let $F(x)$ be a real differentiable function such that $F'(x)$ is monotonic and $F'(x)\geq m>0$ or 
$F'(x)\leq -m<0$ for $a\leq x\leq b$, and $G$ is a positive, monotonic function for $a\leq x \leq b$ such that 
$|G(x)|\leq G$.  Then 
\begin{equation}                                                          \label{lem4011}
\left|\int_{a}^{b}G(x){\rm e}^{iF(x)}dx\right| \leq 4~\frac{G}{m}. 
\end{equation}
\end{lem}
\begin{proof}
The  first derivative test of exponential integral follows from  (2.3) in   \cite{I}. 
\end{proof}
\begin{lem}        
\label{lem70}
Suppose that $a_n$   is an  arbitrary complex  sequence.  
For any sufficiently large  number $x>1$, we have      
\begin{equation}                                                             \label{Hilbert}
\left|\sum_{\substack{m,n\leq x \\ m\neq n}}\frac{a_{m}{\overline{a_{n}}}}{m-n}\right|
\leq \pi \sum_{n\leq x}|a_{n}|^{2}. 
\end{equation}
\end{lem}
\begin{proof}
The  Hilbert inequality follows from (5.6) in~\cite{I}.   
\end{proof}   
 \section{Proofs of  Theorems \ref{th1} and \ref{th2} }   

\subsection{Proof of  Theorem \ref{th1}  } 
We take $f=\id_{1+a}$, for any fixed negative number $a\ (-1<a<0)$,  
into  \eqref{ikiuchi} to get 
\begin{align}                                     \label{th1-1}
   M_{r}(x;\id_{1+a})   
&=\frac12 \sum_{n\leq x}n^{a}    
+ \frac{1}{r+1}\sum_{n\leq x}\left(\frac{\mu}{\rm id}*{\rm id}_{a}*{\bf 1}\right)(n) \nonumber \\
&\qquad     + \frac{1}{r+1}\sum_{m=1}^{[r/2]} \binom{r+1}{2m}B_{2m}
\sum_{n\leq x}\left(\frac{\mu}{\rm id}*{\rm id}_{a}*\frac{1}{\id_{2m}}\right)(n)  \nonumber \\
&=\frac12 \sum_{n\leq x}n^{a}    
+ \frac{1}{r+1}\sum_{n\leq x}\left(\frac{\mu}{{\rm id}}*\sigma_{a}\right)(n) \nonumber \\
&\qquad    + \frac{1}{r+1}\sum_{m=1}^{[r/2]} \binom{r+1}{2m} B_{2m}
\sum_{dl\leq x} \left(\frac{\mu}{\id}*\frac{1}{\id_{2m}}\right)(d) l^{a} \nonumber \\
&:= I_{1}+ I_{2} +I_{3},
\end{align} 
say. From~\eqref{lem1-apo}  we get    
$$                                                                            
I_{1}= \frac{1}{2(1+a)}x^{1+a} + \frac{\zeta(-a)}{2} + \Ocal_{a}\left(x^{a}\right). 
$$  
We use  \eqref{lem2-sigma1} and  \eqref{lem2-la} to  deduce  the formulas    
\begin{equation*}
 I_{2} 
=\frac{\zeta(1-a)}{(r+1)\zeta(2)}x 
+ \frac{\zeta(1+a)}{(1+a)(r+1)\zeta(2+a)} x^{1+a} 
 +  \frac{1}{r+1} \sum_{d\leq x }\frac{\mu(d)}{d} \Delta_{a}\left(\frac{x}{d}\right)  
  + \Ocal_{r,a}\left(\delta(x)\right),   
\end{equation*} 
and 
\begin{equation*}
 I_{3} 
=\frac{x^{1+a}}{(1+a)(r+1)\zeta(2+a)} 
  \sum_{m=1}^{[r/2]}\binom{r+1}{2m}B_{2m} \zeta(2m+a+1)+ \Ocal_{r,a}\left(1\right).    
\end{equation*} 
On combining the above formulas  with~\eqref{th1-1}, we complete the proof of   Theorem~\ref{th1}.   
\subsection{Proof of Theorem~\ref{th2} }
We take $f=\phi_{1+a}$, for $-1<a<0$, into  \eqref{ikiuchi} to get 
\begin{align}    
\label{th2-1}
 M_{r}(x;\phi_{1+a})   
&=\frac12 \sum_{n\leq x}\frac{\phi_{1+a}(n)}{n}    
+ \frac{1}{r+1}\sum_{dl\leq x} \frac{(\phi_{1+a}*\mu)(d)}{d}  \nonumber \\
&\qquad     + \frac{1}{r+1}\sum_{m=1}^{[r/2]} \binom{r+1}{2m} B_{2m}
\sum_{dl\leq x} \frac{(\phi_{1+a}*\mu)(d)}{d} \frac{1}{l^{2m}}  \nonumber \\
&:= J_{1}+ J_{2} +J_{3},
\end{align} 
say.  Using \eqref{lem1-apo}, we find that     
\begin{eqnarray*}                       
J_{1}
&=&\frac12 \sum_{l\leq x}\frac{\mu(l)}{l}\left(\frac{1}{1+a}\left(\frac{x}{l}\right)^{1+a}+\zeta(-a)+\Ocal_{a}\left(\left(\frac{x}{l}\right)^{a}\right)\right) 
\\
&=& \frac{1}{2(1+a)\zeta(2+a)}x^{1+a} + \Ocal_{a}\left(1\right). 
\end{eqnarray*}
We use \eqref{lem2-mu} and \eqref{lem2-l2m} to  deduce the formulas    
\begin{multline*}
 J_{2} 
=\frac{\zeta(1-a)}{(r+1)\zeta(2)^2}x 
+ \frac{\zeta(1+a)}{(1+a)(r+1)\zeta(2+a)^2} x^{1+a}  \\
 + \frac{1}{r+1} \sum_{d\leq x }\frac{(\mu*\mu)(d)}{d} \Delta_{a}\left(\frac{x}{d}\right)   + \Ocal_{r,a}\left((\log x)^2\right),   
\end{multline*} 
and 
\begin{equation*}
 J_{3} 
=\frac{x^{1+a}}{(1+a)(r+1)\zeta(2+a)^2} 
  \sum_{m=1}^{[r/2]}\binom{r+1}{2m}B_{2m} \zeta(2m+a+1)+ \Ocal_{r,a}\left((\log x)^2\right).    
\end{equation*} 
Substituting  the above into \eqref{th2-1}, we complete the proof of   \eqref{sum_th2}.  \\

Now, we replace $f$ by $\psi_{1+a}$,  for $-1<a<0$,  in \eqref{ikiuchi} to get  
\begin{align}                                                                                           \label{th2-2}
   M_{r}(x;\psi_{1+a})   
&=\frac12 \sum_{n\leq x}\frac{\psi_{1+a}(n)}{n}    
+ \frac{1}{r+1}\sum_{dl\leq x} \frac{(\psi_{1+a}*\mu)(d)}{d}  \nonumber \\
&\qquad     + \frac{1}{r+1}\sum_{m=1}^{[r/2]} \binom{r+1}{2m} B_{2m}
\sum_{dl\leq x} \frac{(\psi_{1+a}*\mu)(d)}{d} \frac{1}{l^{2m}}  \nonumber \\
&:= L_{1}+ L_{2} +L_{3},
\end{align} 
say. Using \eqref{lem1-apo},  we find that      
\begin{equation*}                             
L_{1}= \frac{\zeta(2+a)}{2(1+a)\zeta(4+2a)}x^{1+a} + \Ocal_{a}\left(1\right). 
\end{equation*}
Applying \eqref{lem2-psi} and  \eqref{lem2-psi2m}, we get      
\begin{multline*}
 L_{2} 
=\frac{\zeta(1-a)}{(r+1)\zeta(4)}x 
+ \frac{\zeta(1+a)}{(1+a)(r+1)\zeta(4+2a)} x^{1+a}   \\
  + \frac{1}{r+1} \sum_{d\leq x }\frac{(\mu*\mu)(d)}{d} \Delta_{a}\left(\frac{x}{d}\right)   + \Ocal_{r,a}\left((\log x)^2\right),   
\end{multline*} 
and 
\begin{equation*}
 L_{3} 
=\frac{x^{1+a}}{(1+a)(r+1)\zeta(4+2a)} 
  \sum_{m=1}^{[r/2]}\binom{r+1}{2m}B_{2m} \zeta(2m+a+1)   + \Ocal_{r,a}\left((\log x)^2\right).    
\end{equation*} 
Substituting  the above into \eqref{th2-2}, we complete the proof of   \eqref{sum_th22}.  

\section{Proof of Theorem \ref{th5}} %


We take $f={\rm id}_{s+a}$ into \eqref{kiuchi} to  obtain 
\begin{align}                                                                                                        \label{th5-1}
   M_{r}^{(s)}(x;{\rm id}_{s+a}) 
&= \frac{1}{2}\sum_{n\leq x}n^{a} 
+  \frac{1}{r+1} \sum_{dl\leq x} \frac{\mu(d)}{d^s}\sigma_{a}(l)    \\
& \qquad    
+ \frac{1}{r+1} \sum_{m=1}^{[r/2]} \binom{r+1}{2m} B_{2m} 
  \sum_{dl\leq x} \left(\frac{\mu}{{\rm id}_{s}}*\frac{1}{{\rm id}_{2ms}}\right)(d) l^{a}   \nonumber \\
&:= I_{1} + I_{2} + \frac{1}{r+1} \sum_{m=1}^{[r/2]} \binom{r+1}{2m} B_{2m} I_{3},   \nonumber 
\end{align} 
say. From (\ref{lem1-apo}) we have  
\begin{equation}                                                                                                         \label{th5-2}
I_{1}= \frac{1}{2(1+a)}x^{1+a} +\frac{\zeta(-a)}{2} + \Ocal_{a}\left(x^{a}\right).                                        
\end{equation}
For $I_{2}$,  we  use \eqref{Delta-a} and the formula 
$$
\sum_{d\leq x}\frac{\mu(d)}{d^{s+1}} = \frac{1}{\zeta(s+1)} + \Ocal_{s}\left(x^{-s}\right),
$$
to get    
\begin{align}                                                                                                           \label{th5-3}    
 I_{2} 
&= \frac{\zeta(1-a)}{r+1} x \sum_{d\leq x}\frac{\mu(d)}{d^{s+1}} 
 + \frac{\zeta(1+a)}{(r+1)(1+a)}x^{1+a} \sum_{d\leq x} \frac{\mu(d)}{d^{s+1+a}} - \frac{\zeta(-a)}{2(r+1)}\sum_{d\leq x} \frac{\mu(d)}{d^{s}} 
 \nonumber  \\
&+ \frac{1}{r+1}\sum_{d\leq x}\frac{\mu(d)}{d^s}\Delta_{a}\left(\frac{x}{d}\right)     \nonumber  \\
&= \frac{\zeta(1-a)}{(r+1)\zeta(s+1)} x  
 + \frac{\zeta(1+a)}{(r+1)(1+a)\zeta(s+1+a)}x^{1+a}   - \frac{\zeta(-a)}{2(r+1)\zeta(s)}  \nonumber  \\
& + \frac{1}{r+1}\sum_{d\leq x}\frac{\mu(d)}{d^s}\Delta_{a}\left(\frac{x}{d}\right)  +\Ocal_{r,s}\left(x^{1-s}\right).  
\end{align}
As for $I_{3}$,  we use  \eqref{lem1-apo} to get       
\begin{eqnarray}                                                                                                        \label{th5-4}
I_{3} 
&=& \frac{x^{1+a}}{1+a}\sum_{dl\leq x}\frac{\mu(d)}{d^{s+a+1}}\frac{1}{l^{2ms+a+1}} 
 + \zeta(-a)\sum_{dl\leq x}\frac{\mu(d)}{d^{s}}\frac{1}{l^{2ms}}  + \Ocal_{s,a}\left(x^{a}\sum_{dl\leq x}\frac{1}{d^{s+a}}\frac{1}{l^{2ms+a}}\right) \nonumber 
\\&=& \frac{\zeta(2ms+a+1)}{(1+a)\zeta(s+a+1)} x^{1+a} + \frac{\zeta(-a)\zeta(2ms)}{\zeta(s)} +  \Ocal_{s,a}\left(x^{a}\right).  
\end{eqnarray}
Therefore, putting everything together  we  obtain  \eqref{sum_th5}.  

\section{Proofs of Theorems \ref{th71} and \ref{th72}} 

From \eqref{Sumaia-K}, we have  
\begin{eqnarray}     
\label{Sumaia-KKK}
 M_{r}^{(s)}(x;{\rm id}_{s}*h)   
 &=& \frac{1}{2}\sum_{dl\leq x}\frac{h(d)}{d^s} 
+ \frac{1}{r+1}\sum_{dl\leq x} \frac{(h*\mu)(d)}{d^s} \tau(l) \nonumber \\&&
+ \frac{1}{r+1}\sum_{m=1}^{[r/2]}\binom{r+1}{2m} B_{2m}\sum_{dl\leq x}\frac{(h*\phi_{s})(d)}{d^s}\frac{1}{l^{2ms}}  \nonumber \\
&:=& I_{1}+I_{2}+I_{3},  
\end{eqnarray}
say.  Since 
$
h(d) = \Ocal\left(d^{\varepsilon}\right),
$ for any small number $\varepsilon>0$,  then  we have   
\begin{equation}   
\label{S-KKK1}
I_{1}
 = \frac{H_{h}(s+1)}{2} x - \sum_{d\leq x}\frac{h(d)}{d^{s}}\vartheta\left(\frac{x}{d}\right) 
 - \frac{H_{h}(s)}{2} + \Ocal_{s}\left(x^{1-s+\varepsilon}\right),                                         
\end{equation} 
where  $\vartheta(x)=x- [x] -\frac{1}{2}$. We use \eqref{Dirichlet} and the identity 
$$
H'_{h}(s+1)=-\sum_{n=1}^{\infty}\frac{h(d)\log d}{d^{s+1}}
$$
to obtain 
\begin{align}                                                                                            \label{S-KKK2}
I_{2}
&= \frac{x\log x}{r+1}\sum_{d\leq x} \frac{(h*\mu)(d)}{d^{s+1}}   
 - \frac{x}{r+1}\sum_{d\leq x} \frac{(h*\mu)(d)}{d^{s+1}} \log d  + \frac{(2\gamma - 1)x}{r+1}\sum_{d\leq x} \frac{(h*\mu)(d)}{d^{s+1}}   
\nonumber \\
&  + \frac{1}{r+1}\sum_{d\leq x} \frac{(h*\mu)(d)}{d^s} \Delta\left(\frac{x}{d}\right) \nonumber \\
&= \frac{H_{h}(s+1)}{(r+1)\zeta(s+1)} x\log x + 
   \left(\frac{H'_{h}(s+1)}{(r+1)\zeta(s+1)}-\frac{H_{h}(s+1)\zeta'(s+1)}{(r+1)\zeta(s+1)^2}\right)x  \nonumber  \\
&+ \frac{2(\gamma -1)H_{h}(s+1)}{(r+1)\zeta(s+1)}x  + \frac{1}{r+1}\sum_{d\leq x} \frac{(h*\mu)(d)}{d^s} \Delta\left(\frac{x}{d}\right) 
 + \Ocal_{s,r}\left(x^{1-s+\varepsilon}\right).
\end{align}
Using the formula  
\begin{equation*}                                                                                                  
\sum_{l\leq x}\sigma_{-2ms}(l)  = 
\zeta(2ms+1) x -\frac{\zeta(2ms)}{2}  + \Ocal_{m,s}\left(x^{1-2ms}\right),
\end{equation*}
we have 
\begin{align}                                                                                            \label{S-KKK3}
I_{3}
&= \frac{1}{r+1}\sum_{m=1}^{[r/2]}\binom{r+1}{2m}B_{2m} \sum_{d\leq x} \frac{(h*\mu)(d)}{d^{s}}\sum_{l\leq x/d}\sigma_{-2ms}(l) \nonumber \\ 
&= \frac{x}{r+1}\sum_{m=1}^{[r/2]}\binom{r+1}{2m}B_{2m} \zeta(2ms+1) \sum_{d\leq x} \frac{(h*\mu)(d)}{d^{s+1}}  \nonumber \\
&- \frac{1}{2(r+1)}\sum_{m=1}^{[r/2]}\binom{r+1}{2m}B_{2m}\zeta(2ms) \sum_{d\leq x} \frac{(h*\mu)(d)}{d^{s}}
+  \Ocal_{r,s}\left(x^{1-s+\varepsilon}\right) \nonumber \\
&= \frac{H_{h}(s+1)x}{(r+1)\zeta(s+1)} \sum_{m=1}^{[r/2]}\binom{r+1}{2m}B_{2m} \zeta(2ms+1)   \nonumber \\
&- \frac{H_{h}(s)}{2(r+1)\zeta(s)} \sum_{m=1}^{[r/2]}\binom{r+1}{2m}B_{2m}\zeta(2ms) 
+  \Ocal_{r,s}\left(x^{1-s+\varepsilon}\right). 
\end{align}
Substituting \eqref{S-KKK1}, \eqref{S-KKK2}  and \eqref{S-KKK3}  into \eqref{Sumaia-KKK}, 
we get  the formula \eqref{sum_th770}. Similarly, we  can obtain  the formula \eqref{sum_th777}. 
\section{Proofs of  Theorems  \ref{th3} and \ref{th4} }    
\subsection{Proof of  Theorem \ref{th3} } 
To get the mean value formula of $K_{r}(x;{\rm id}_{1+a})$, 
we follow the  method  used  in~\cite{K1}.    
Let $a$ be any fixed  number satisfying $- \frac14 < a <0$. 
We recall that   
\begin{equation}                                                                                                     \label{Lr1}
K_{r}(x;\id_{1+a}) =  \frac{1}{r+1} \sum_{n\leq x} \frac{\mu(n)}{n}\Delta_{a}\left(\frac{x}{n}\right)
        +  \Ocal_{r,a}\left(1\right).
\end{equation}
We want to find  the corresponding formula for the  integral 
$$
\int_{T}^{2T}K_{r}(x;{\rm id}_{1+a})^{2}dx.
$$  
We take  $y=T^{1-\varepsilon}$ for any small  number $\varepsilon>0$.
From \eqref{Lr1}, we have 
\begin{eqnarray*}
K_{r}(x;\id_{1+a}) &=& \frac{1}{r+1}
\left( \sum_{n\leq y} \frac{\mu(n)}{n}\Delta_{a}\left(\frac{x}{n}\right) 
+ \sum_{y<n\leq x} \frac{\mu(n)}{n}\Delta_{a}\left(\frac{x}{n}\right) \right) 
+ \Ocal_{r,a}\left(1\right)  \\
&:=& \frac{1}{r+1}\left(F_{1,a}(x)+F_{2,a}(x)\right) +  \Ocal_{r,a}\left(1\right),
\end{eqnarray*}
say. 
For $F_{2,a}(x)$, it is easy to see that  
$
F_{2,a}(x) \ll \sum_{y<n\leq x} \frac{1}{n} \Delta_{a}\left(\frac{x}{n}\right) 
\ll    T^{\varepsilon}  
$ 
where we used the estimate 
$
\Delta_{a}\left(\frac{x}{n}\right) \ll   \left(\frac{x}{y}\right)^{\frac{1+a}{3}+\varepsilon} \ll T^{\varepsilon} 
$  
with 
$
T\leq x\leq 2T.
$
As for $F_{1,a}(x)$,  we   take $N=y$ into \eqref{lem3011} 
to deduce     
$$
F_{1,a}(x) = G_{a}(x) + \Ocal_{a}\left(T^{\varepsilon}\right), 
$$ 
where 
$$
G_{a}(x) = \frac{x^{\frac14 +\frac{a}{2}}}{\sqrt{2}\pi} \sum_{m\leq y}\frac{\mu(m)}{m^{\frac{5}{4}+\frac{a}{2}}} 
 \sum_{n\leq y} \frac{\sigma_{a}(n)}{n^{\frac{3}{4}+\frac{a}{2}}}\cos\left(4\pi\sqrt{\frac{nx}{m}}-\frac{\pi}{4}\right).  
$$
Then,  we have 
\begin{equation}                                                                                                               \label{III}
K_{r}(x;\id_{1+a})  = \frac{1}{r+1} G_{a}(x) + \Ocal_{r,a}\left(T^{\varepsilon}\right).
\end{equation}
Now, we  consider the mean square formula of $G_{a}(x)$,  that is  
\begin{eqnarray*}                                                                                            
G_{a}(x)^{2} &=& 
\frac{x^{\frac12 +a}}{2\pi^2}\sum_{m_{1},m_{2}\leq y}\frac{\mu(m_1)\mu(m_2)}{(m_{1}m_{2})^{\frac{5}{4}+\frac{a}{2}}}
                      \sum_{n_{1},n_{2}\leq y}\frac{\sigma_{a}(n_{1})\sigma_{a}(n_{2})}{(n_{1}n_{2})^{\frac{3}{4}+\frac{a}{2}}}  \times \\&&  
\times \cos\left(4\pi\sqrt{\frac{n_{1}x}{m_{1}}}-\frac{\pi}{4}\right) \cos\left(4\pi\sqrt{\frac{n_{2}x}{m_{2}}}-\frac{\pi}{4}\right) \\ 
& :=& S_{1,a}(x) + S_{2,a}(x) + S_{3,a}(x), 
\end{eqnarray*}
say,  where  
\begin{equation*}
S_{1,a}(x) = \frac{x^{\frac{1}{2}+a}}{4\pi^2}\sum_{\substack{m_{1},m_{2},n_{1},n_{2} \leq y  \\ n_{1}m_{2} = n_{2}m_{1}}}  
            \frac{\mu(m_1)\mu(m_2)}{(m_{1}m_{2})^{\frac{5}{4}+\frac{a}{2}}}
\frac{\sigma_{a}(n_1)\sigma_{a}(n_2)}{(n_{1}n_{2})^{\frac{3}{4}+\frac{a}{2}}},    
\end{equation*}
\begin{multline*}
S_{2,a}(x)=  \frac{x^{\frac{1}{2}+a}}{4\pi^2}\sum_{\substack{m_{1},m_{2},n_{1},n_{2} \leq y  \\ n_{1}m_{2} \neq  n_{2}m_{1}}}  
 \frac{\mu(m_1)\mu(m_2)}{(m_{1}m_{2})^{\frac{5}{4}+\frac{a}{2}}}
 \frac{\sigma_{a}(n_1)\sigma_{a}(n_2)}{(n_{1}n_{2})^{\frac{3}{4}+\frac{a}{2}}}  \times \\
 \times   \cos\left(4\pi\sqrt{x}\left(\sqrt{\frac{n_1}{m_1}}-\sqrt{\frac{n_2}{m_2}}\right)\right),           \end{multline*}
and
\begin{multline*}
S_{3,a}(x)=  \frac{x^{\frac{1}{2}+a}}{4\pi^2}\sum_{m_{1},m_{2},n_{1},n_{2} \leq y}  
 \frac{\mu(m_1)\mu(m_2)}{(m_{1}m_{2})^{\frac{5}{4}+\frac{a}{2}}}
 \frac{\sigma_{a}(n_1)\sigma_{a}(n_2)}{(n_{1}n_{2})^{\frac{3}{4}+\frac{a}{2}}}  \times  \\
 \times   \sin\left(4\pi\sqrt{x}\left(\sqrt{\frac{n_1}{m_1}}+\sqrt{\frac{n_2}{m_2}}\right)\right).                        
\end{multline*} 
Firstly, we have 
\begin{equation}                                                            \label{SSS1} 
\int_{T}^{2T}S_{1,a}(x)dx = \frac{B_{a}(T)}{2(3+2a)\pi^2}\left((2T)^{\frac{3}{2}+a} - T^{\frac{3}{2}+a}\right),   
\end{equation}
where 
$$
B_{a}(T) =\sum_{\substack{m_{1}, m_{2}, n_{1}, n_{2} \leq y  \\  n_{1}m_{2} = n_{2}m_{1}}}  
 \frac{\mu(m_1)\mu(m_2)}{(m_{1}m_{2})^{\frac{5}{4}+\frac{a}{2}}}  \frac{\sigma_{a}(n_1)\sigma_{a}(n_2)}{(n_{1}n_{2})^{\frac{3}{4}+\frac{a}{2}}}.  
$$
In order to evaluate $B_{a}(T)$, one can write   
\begin{eqnarray*}
B_{a}(T) &=&  \sum_{\substack{m_{1},m_{2},n_{1},n_{2} \leq y  \\ n_{1}m_{2} = n_{2}m_{1}}}  
            \frac{\mu(m_1)\mu(m_2)(m_{1}m_{2})^{-\frac{1+a}{2}}\sigma_{a}(n_1)\sigma_{a}(n_2)}{(n_{1}m_{2}n_{2}m_{1})^{3/4}} \\
     & =& \sum_{n\leq y^2} \frac{h_{a}^{2}(n;y)}{n^{\frac{3}{2}}},  
\end{eqnarray*}
where
$$
h_{a}(n;y)= \sum_{\substack{n=ml \\ m,l\leq y}} \frac{\mu(m)}{m^{\frac{1+a}{2}}} \sigma_{a}(l).  
$$
Now, let 
$$
h_{a}(n)= \sum_{n=ml} \frac{\mu(m)}{m^{\frac{1+a}{2}}}\sigma_{a}(l), 
\qquad 
\text{and}
\qquad  
{\widetilde{h_{a}}}(n)= \sum_{n=ml} \frac{1}{m^{\frac{1+a}{2}}} \sigma_{a}(l).
$$
It follows that 
$
h_{a}(n;y)=h_{a}(n)
$ for $n\leq y$,   
$
|h_{a}(n;y)|\leq {\widetilde{h_{a}}}(n)
$
 and  
$
|h_{a}(n)|\leq {\widetilde{h_{a}}}(n)
$
 for $n\geq 1$. 
Since 
$
{\widetilde{h_{a}}}(n)\ll n^{\varepsilon}
$ 
and 
$$
 {\widetilde{h_{a}}}(p) = p^{-\frac12 -\frac{a}{2}}\sum_{d|p}d^{\frac12 +\frac{a}{2}}\sigma_{a}(d)
 =  1 + p^{a} +p^{-\frac12 -\frac{a}{2}} < 3,
$$
we use  Shiu's theorem (see~\cite{S}) and the prime number theorem  to obtain 
\begin{eqnarray*}  
     \sum_{n\leq x}{\widetilde{h_{a}}}(n)^{2} 
&\ll& \frac{x}{\log x} {\rm exp}\left(\sum_{p\leq x}\frac{{\widetilde{h_{a}}}(p)^{2}}{p}\right)   \\
&  \ll& \frac{x}{\log x} {\rm exp}\left(\sum_{p\leq x}\frac{9}{p}\right)  
\ll x^{}(\log x)^{8}.  
\end{eqnarray*}  
Hence   
\begin{eqnarray*}
B_{a}(T) 
& =& \sum_{n=1}^{\infty}\frac{h_{a}(n)^{2}}{n^{3/2}} + \Ocal_{a}\left(\sum_{n>y}\frac{{\widetilde{h_{a}}}(n)^{2}}{n^{3/2}}\right)  \\
&=& \sum_{n=1}^{\infty}\frac{h_{a}(n)^{2}}{n^{3/2}} + \Ocal_{a}\left(y^{-1/2}(\log y)^{8}  \right)  \\
&=& \sum_{n=1}^{\infty}\left(\left(\frac{\mu}{{\rm id}_{\frac{1+a}{2}}}*\sigma_{a}\right)(n)\right)^{2}\frac{1}{n^{3/2}} + \Ocal_{a}\left(T^{-1/2+\varepsilon} \right). 
\end{eqnarray*}
Thus, using \eqref{SSS1} and the above result, we obtain  
\begin{multline}                                                                                                             \label{SSS11} 
\int_{T}^{2T}S_{1,a}(x)^{}dx = 
\frac{1}{2(3+2a)\pi^2}\left(\sum_{n=1}^{\infty}\frac{h_{a}(n)^{2}}{n^{3/2}}\right)\left((2T)^{\frac{3}{2}+a} - T^{\frac{3}{2}+a}\right) 
+ \Ocal_{a}\left(T^{1+a+\varepsilon}\right).      
\end{multline}
Next, we consider the integral of $S_{2,a}(x)$. Using \eqref{lem4011}, the estimates
\begin{eqnarray*}
&&\sum_{n\leq x}(\sigma_{a}*{\bf 1})(n) \ll x \log x,
\\&&
\sum_{n\leq x}(\sigma_{a}*{\bf 1})(n)^{2} \ll x(\log x)^{8}, 
\end{eqnarray*}
and the Hilbert inequality \eqref{Hilbert}, we obtain  
\begin{eqnarray}
\label{SSS2} 
\int_{T}^{2T}  S_{2,a}(x)dx  &\ll& T^{1+a} \sum_{\substack{m_{1},m_{2},n_{1},n_{2} \leq y  \\ n_{1}m_{2} \neq  n_{2}m_{1}}}  
            \frac{1}{(m_{1}m_{2})^{\frac{5}{4}+\frac{a}{2}}}\frac{\sigma_{a}(n_1)\sigma_{a}(n_2)}{(n_{1}n_{2})^{\frac{3}{4}+\frac{a}{2}}}  
            \frac{1}{ \left|\sqrt{\frac{n_{1}}{m_{1}}}-\sqrt{\frac{n_{2}}{m_{2}}}\right|^{}}       \nonumber \\
&\ll& T^{1+a} \sum_{\substack{m_{1},m_{2},n_{1},n_{2} \leq y  \\ n_{1}m_{2} \neq  n_{2}m_{1}}}  
            \frac{\sigma_{a}(n_1)\sigma_{a}(n_2)}{(n_{2}m_{1}n_{1}m_{2})^{\frac{3}{4}+\frac{a}{2}}}   
            \frac{1}{ \left|\sqrt{n_{1}m_{2}}-\sqrt{n_{2}m_{1}}\right|^{}}       \nonumber \\
&\ll& T^{1+a} \sum_{\substack{m,n \leq y^2  \\ |\sqrt{n}-\sqrt{m}|\geq \frac12 (mn)^{\frac{1}{4}}}}  
            \frac{(\sigma_{a}*{\bf 1}_{})(n)(\sigma_{a}*{\bf 1}_{})(m)}{(nm)^{\frac{3}{4}+\frac{a}{2}}}   
            \frac{1}{\left|\sqrt{n}-\sqrt{m}\right|^{}}   \nonumber \\
&&+  T^{1+a} \sum_{\substack{m,n \leq y^2  \\ 0<|\sqrt{n}-\sqrt{m}|\leq  \frac12 (mn)^{\frac{1}{4}}}}  
            \frac{(\sigma_{a}*{\bf 1}_{})(n)(\sigma_{a}*{\bf 1})(m)}{(nm)^{\frac{3}{4}+\frac{a}{2}}} 
            \frac{1}{\left|\sqrt{n}-\sqrt{m}\right|}    \nonumber \\
&\ll& T^{1+a} \left(\sum_{n \leq y^2}\frac{(\sigma_{a}*\1)(n)}{n^{1+\frac{a}{2}}}\right)^{2} 
+    T^{1+a} \sum_{\substack{n,m  \leq y^2 \\ m\neq n}} 
      \frac{(\sigma_{a}*{\1})(n)(\sigma_{a}*\1)(m)}{(nm)^{\frac{1}{2}+\frac{a}{2}}|n-m|} \nonumber \\
&\ll& T^{1-a+\varepsilon} + 
     T^{1+a} \sum_{n \leq y^2} \frac{(\sigma_{a}*\1)(n)^{2}}{n^{1+a}} \nonumber \\
&\ll& T^{1-a+\varepsilon}.                                          
\end{eqnarray}
Lastly, we evaluate the integral $S_{3}(x)$. Using the formula 
$$
\sum_{n\leq x}\frac{\sigma_{a}(n)}{n^{1+\frac{a}{2}}} \ll x^{-\frac{a}{2}}
$$
and \eqref{lem4011} again, we obtain 
\begin{eqnarray}
\label{SSS3} 
\int_{T}^{2T}S_{3,a}(x)dx
&\ll& T^{1+a} \sum_{m_{1},m_{2},n_{1},n_{2} \leq y }  
       \frac{1}{(m_{1}m_{2})^{\frac{5}{4}+\frac{a}{2}}} \frac{\sigma_{a}(n_1)\sigma_{a}(n_2)}{(n_{1}n_{2})^{\frac{3}{4}+\frac{a}{2}}} 
       \frac{1}{\left|\sqrt{\frac{n_{1}}{m_{1}}}+\sqrt{\frac{n_{2}}{m_{2}}}\right|}
       \nonumber \\
&\ll& T^{1+a} \sum_{m_{1},m_{2},n_{1},n_{2} \leq y }    
            \frac{\sigma_{a}(n_1)\sigma_{a}(n_2)}{(m_{1}m_{2})^{\frac{5}{4}+\frac{a}{2}}(n_{1}n_{2})^{\frac{3}{4}+\frac{a}{2}}}   
            \frac{1}{ \left(\frac{n_{1}}{m_{1}}\cdot \frac{n_{2}}{m_{2}}\right)^{\frac{1}{4}}}         \nonumber \\
&\ll& T^{1+a} \sum_{m_{1},m_{2},n_{1},n_{2} \leq y }   
            \frac{\sigma_{a}(n_1)\sigma_{a}(n_2)}{(m_{1}m_{2}n_{1}n_{2})^{1+\frac{a}{2}}}   \nonumber  \\
&\ll& T^{1+a} 
     \left(\sum_{m\leq y}\frac{1}{m^{1+\frac{a}{2}}}\right)^{2} 
     \left(\sum_{n\leq y}\frac{\sigma_{a}(n)}{n^{1+\frac{a}{2}}}\right)^{2} 
\nonumber \\ &\ll& T^{1-a+\varepsilon}.                                                                                                  
\end{eqnarray} 
From \eqref{SSS11}--\eqref{SSS3},  we have    
\begin{equation}                                                           \label{SSS4} 
\int_{T}^{2T}G^{2}_{a}(x)dx  
=  \frac{1}{2(3+2a)\pi^2}\left(\sum_{n=1}^{\infty}\frac{h_{a}(n)^{2}}{n^{3/2}}\right)\left((2T)^{\frac{3}{2}+a} - T^{\frac{3}{2}+a}\right) 
 + \Ocal_{a}\left(T^{1-a+\varepsilon}\right). 
\end{equation} 
We use \eqref{III}, \eqref{SSS4} and the Cauchy--Schwarz inequality to obtain the desired result. 
\subsection{Proof   of  Theorem \ref{th4}}
We proceed  as in the proof of  Theorem \ref{th3}.    
We let $y=T^{1-\varepsilon}$  and  $N=y$ in  (\ref{lem3011}).  Then,  we find       
\begin{equation}                                                                                                               \label{ZZI}
K_{r}(x;\phi_{1+a}) = \frac{1}{r+1} H_{a}(x) + \Ocal_{r,a}\left(T^{\varepsilon}\right),  
\end{equation}
where  
$$
H_{a}(x) = \frac{x^{\frac{1}{4}+\frac{a}{2}}}{\sqrt{2}\pi} \sum_{m\leq y}\frac{(\mu*\mu)(m)}{m^{\frac{5}{4}+\frac{a}{2}}} 
 \sum_{n\leq y} \frac{\sigma_{a}(n)}{n^{\frac{3}{4}+\frac{a}{2}}}\cos\left(4\pi\sqrt{\frac{nx}{m}}-\frac{\pi}{4}\right).  
$$
Now, we consider the mean square formula of $H_{a}(x)$,  that  is   
\begin{eqnarray*}                                                                                            
H_{a}(x)^{2} &=& 
\frac{x^{\frac12 +a}}{2\pi^2}\sum_{m_{1},m_{2}\leq y}\frac{(\mu*\mu)(m_1)(\mu*\mu)(m_2)}{(m_{1}m_{2})^{\frac{5}{4}+\frac{a}{2}}}
                      \sum_{n_{1},n_{2}\leq y}\frac{\sigma_{a}(n_{1})\sigma_{a}(n_{2})}{(n_{1}n_{2})^{\frac{3}{4}+\frac{a}{2}}}  \times  \\ 
&& \times \cos\left(4\pi\sqrt{\frac{n_{1}x}{m_{1}}}-\frac{\pi}{4}\right) \cos\left(4\pi\sqrt{\frac{n_{2}x}{m_{2}}}-\frac{\pi}{4}\right) \\ 
& :=& Y_{1,a}(x) + Y_{2,a}(x) + Y_{3,a}(x), 
\end{eqnarray*}
say,  where  
\begin{equation*}
Y_{1,a}(x) = \frac{x^{\frac{1}{2}+a}}{4\pi^2}\sum_{\substack{m_{1},m_{2},n_{1},n_{2} \leq y  \\ n_{1}m_{2} = n_{2}m_{1}}}  
   \frac{(\mu*\mu)(m_1)(\mu*\mu)(m_2)}{(m_{1}m_{2})^{\frac{5}{4}+\frac{a}{2}}}
 \frac{\sigma_{a}(n_1)\sigma_{a}(n_2)}{(n_{1}n_{2})^{\frac{3}{4}+\frac{a}{2}}},   
\end{equation*}
\begin{multline*}
Y_{2,a}(x)=   \frac{x^{\frac{1}{2}+a}}{4\pi^2}\sum_{\substack{m_{1},m_{2},n_{1},n_{2} \leq y  \\ n_{1}m_{2} \neq  n_{2}m_{1}}}  
 \frac{(\mu*\mu)(m_1)(\mu*\mu)(m_2)}{(m_{1}m_{2})^{\frac{5}{4}+\frac{a}{2}}}
 \frac{\sigma_{a}(n_1)\sigma_{a}(n_2)}{(n_{1}n_{2})^{\frac{3}{4}+\frac{a}{2}}}  \times  \\
 \times   \cos\left(4\pi\sqrt{x}\left(\sqrt{\frac{n_1}{m_1}}-\sqrt{\frac{n_2}{m_2}}\right)\right),
        \end{multline*}
and
\begin{multline*}
Y_{3,a}(x)=  \frac{x^{\frac{1}{2}+a}}{4\pi^2}\sum_{m_{1},m_{2},n_{1},n_{2} \leq y}  
            \frac{(\mu*\mu)(m_1)(\mu*\mu)(m_2)}{(m_{1}m_{2})^{\frac{5}{4}+\frac{a}{2}}}
            \frac{\sigma_{a}(n_1)\sigma_{a}(n_2)}{(n_{1}n_{2})^{\frac{3}{4}+\frac{a}{2}}}  \times  \\
  \times   \sin\left(4\pi\sqrt{x}\left(\sqrt{\frac{n_1}{m_1}}+\sqrt{\frac{n_2}{m_2}}\right)\right).                        
\end{multline*} 
Firstly, we have 
\begin{equation}                                                                                                             \label{SS1} 
\int_{T}^{2T}Y_{1,a}(x)^{}dx = \frac{Q_{a}(T)}{2(3+2a)\pi^2}\left((2T)^{\frac{3}{2}+a} - T^{\frac{3}{2}+a}\right),   
\end{equation}
where 
$$
Q_{a}(T) =\sum_{\substack{m_{1}, m_{2}, n_{1}, n_{2} \leq y  \\ n_{1}m_{2} = n_{2}m_{1}}}  
            \frac{(\mu*\mu)(m_1)(\mu*\mu)(m_2)}{(m_{1}m_{2})^{\frac{5}{4}+\frac{a}{2}}}  
            \frac{\sigma_{a}(n_1)\sigma_{a}(n_2)}{(n_{1}n_{2})^{\frac{3}{2}+\frac{a}{2}}}.  
$$
To evaluate $Q_{a}(T)$, one can write   
\begin{eqnarray*}
Q_{a}(T) &=&  \sum_{\substack{m_{1},m_{2},n_{1},n_{2} \leq y  \\ n_{1}m_{2} = n_{2}m_{1}}}  
            \frac{(\mu*\mu)(m_1)(\mu*\mu)(m_2)\sigma_{a}(n_1)\sigma_{a}(n_2)(m_{1}m_{2})^{-\frac{1+a}{2}}}{(n_{1}m_{2}n_{2}m_{1})^{3/4}} \\
     & =& \sum_{n\leq y^2} \frac{g(n;y)^{2}}{n^{3/2}},  
\end{eqnarray*}
where
$$
g(n;y)= \sum_{\substack{n=ml \\ m,l\leq y}} \frac{(\mu*\mu)(m)}{m^{\frac{1+a}{2}}} \sigma_{a}(l).  
$$
Now, let 
$$
g_{0}(n) := \sum_{n=ml} \frac{(\mu*\mu)(m)}{m^{\frac{1+a}{2}}}\sigma_{a}(l) 
$$
and
$$ 
g_{1}(n)= \sum_{n=ml} \frac{\tau(m)}{m{\frac{1+a}{2}}} \sigma_{a}(l).
$$
It follows that 
$g(n;y)=g_{0}(n)$ for $n\leq y$,    
$|g(n;y)|\leq g_{1}(n)$  and  $|g_{0}(n)|\leq g_{1}(n)$  for $n\geq 1$. 
Using   the partial sum (see~\cite[Shiu's theorem]{S})   
\begin{equation*}
\sum_{n\leq x}g_{1}^{2}(n) \ll \frac{x}{\log x} {\rm exp}\left(\sum_{p\leq x}\frac{g_{1}^{2}(p)}{p}\right)   
\ll x^{1+\varepsilon},
\end{equation*}  
 we have 
\begin{equation*}
Q_{a}(T)  = \sum_{n=1}^{\infty}\frac{g_{0}^{2}(n)}{n^{3/2}} + \Ocal_{a}\left(y^{-1/2+\varepsilon}\right).   
\end{equation*}
Hence
\begin{equation}                                                                                                             \label{SS11} 
\int_{T}^{2T}Y_{1,a}(x)^{}dx 
= \frac{1}{6\pi^2}\left(\sum_{n=1}^{\infty}\frac{g_{0}^{2}(n)}{n^{3/2}} \right) \left((2T)^{3/2} - T^{3/2}\right)  + 
  \Ocal_{a}\left(T^{1+a+\varepsilon}\right).   
\end{equation}
Next, we shall consider the integrals  of $Y_{2}(x)$ and $Y_{3}(x)$. 
Due to the method of proofs of \eqref{SSS2} and \eqref{SSS3}, we use \eqref{lem4011}, \eqref{Hilbert} and the inequalities 
$$
\sum_{n\leq x}(\sigma_{a}*\tau)(n) \ll x^{1+\varepsilon}
\qquad {\rm and} \qquad 
\sum_{n\leq x}(\sigma_{a}*\tau)(n)^{2} \ll x^{1+\varepsilon}
$$
to  obtain  
\begin{eqnarray}                                                                                                            \label{SS2} 
 \int_{T}^{2T}Y_{2,a}(x)dx
&\ll& T^{1+a} \sum_{\substack{m_{1},m_{2},n_{1},n_{2} \leq y  \\ n_{1}m_{2} \neq  n_{2}m_{1}}}  
            \frac{\tau(m_1)\tau(m_2)\sigma_{a}(n_1)\sigma_{a}(n_2)}{(n_{2}m_{1}n_{1}m_{2})^{\frac{3}{4}+\frac{a}{2}}}   
            \frac{1}{ \left|\sqrt{n_{1}m_{2}}-\sqrt{n_{2}m_{1}}\right|}       \nonumber \\
&\ll& T^{1+a} \left(\sum_{n \leq y^2}\frac{(\sigma_{a}*\tau)(n)}{n^{1+\frac{a}{2}}}\right)^{2} 
+    T^{1+a} \sum_{\substack{n,m  \leq y^2 \\ m\neq n}} 
      \frac{(\sigma_{a}*\tau)(n)(\sigma_{a}*\tau)(m)}{(nm)^{\frac{1}{2}+\frac{a}{2}}}\frac{1}{|n-m|} \nonumber \\
&\ll& T^{1-a+\varepsilon} + 
     T^{1+a} \sum_{n \leq y^2} \frac{(\sigma_{a}*\tau)(n)^{2}}{n^{1+a}} \nonumber \\
&\ll& T^{1-a+\varepsilon}                                                                                        
\end{eqnarray}
and 
\begin{eqnarray}                                                                                                            \label{SS3} 
\int_{T}^{2T}Y_{3,a}(x)dx   
 &\ll& T^{1+a} \sum_{m_{1},m_{2},n_{1},n_{2} \leq y }   
            \frac{\tau(m_1)\tau(m_2)\sigma_{a}(n_1)\sigma_{a}(n_2)}{(m_{1}m_{2}n_{1}n_{2})^{1+\frac{a}{2}}}   \nonumber  \\
& \ll& T^{1+a} 
     \left(\sum_{m\leq y}\frac{\tau(m)}{m^{1+\frac{a}{2}}}\right)^{2} 
     \left(\sum_{n\leq y}\frac{\sigma_{a}(n)}{n^{1+\frac{a}{2}}}\right)^{2} \nonumber \\
& \ll& T^{1-a+\varepsilon}.                                                                                           
\end{eqnarray} 
From (\ref{SS11})--(\ref{SS3})  we have    
\begin{align}
\int_{T}^{2T}H^{2}_{a}(x)dx  
&=  \frac{1}{2(3+2a)\pi^2} \sum_{n=1}^{\infty}\frac{g_{0}^{2}(n)}{n^{3/2}} \left((2T)^{3/2} - T^{3/2}\right) + \Ocal_{a}\left(T^{1-a+\varepsilon}\right).   \label{SS4} 
\end{align} 
We use (\ref{ZZI}), (\ref{SS4}) and the Cauchy--Schwarz inequality to obtain the desired result. 

In the same manner as the above,   we can prove  the formula (\ref{sum_th44}).  
\section*{Acknowledgement}
The second  author is supported by the Japan Society for the Promotion of Science (JSPS). 
``Overseas researcher under Postdoctoral Fellowship of JSPS" and the Austrian Science
Fund (FWF) : Project F5507-N26, which is part
of the special Research Program  `` Quasi Monte
Carlo Methods : Theory and Application''.  

\medskip\noindent {\footnotesize Isao Kiuchi: Department of Mathematical Sciences, Faculty of Science,
Yamaguchi University, Yoshida 1677-1, Yamaguchi 753-8512, Japan. \\
e-mail: {\tt kiuchi@yamaguchi-u.ac.jp}}

\medskip\noindent {\footnotesize Sumaia Saad Eddin: Graduate School of Mathematics, Nagoya University, Furo-cho, Chikusa-ku, Nagoya, Aichi 464-8602, Japan.\\
e-mail: {\tt saad.eddin@math.nagoya-u.ac.jp}}

\end{document}